\documentclass[11pt]{article}

\usepackage{amsmath}
\usepackage{amssymb}
\usepackage{amsthm}
\usepackage{amsfonts}
\usepackage{enumerate}
\usepackage{graphicx}
\usepackage{epsfig}
\usepackage{epsf}
\usepackage{epstopdf}

\newtheorem{theorem}{Theorem}[section]
\newtheorem{definition}[theorem]{Definition}
\newtheorem{lemma}[theorem]{Lemma}
\newtheorem{proposition}[theorem]{Proposition}

\newtheorem{remark1}[theorem]{Remark}

\newtheorem{corollary}[theorem]{Corollary}

\def\C{\mathbb{C}}

\def\Z{\mathbb{Z}}

\newcommand{\cat}[1]{\textbf{#1}}
\newcommand{\rmod}{\cat{R-mod}}
\newcommand{\amod}{\cat{A-mod}}
\newcommand{\cobtwo}{\cat{2Cob}}
\newcommand{\od}{
\begin{picture}(12, 12) (0,0)
\put(5,-5){\line(0,1){20}}
\put(5,5){\circle*{5}}
\end{picture}}

\newcommand{\nd}{
\begin{picture}(12, 12) (0,0)
\put(5,-5){\line(0,1){20}}
\end{picture}}

\newcommand{\td}{
\begin{picture}(12, 12) (0,0)
\put(5,-5){\line(0,1){20}}
\put(5,8){\circle*{5}}
\put(5,2){\circle*{5}}
\end{picture}}

\title{Generalized Skein Modules of Surfaces}

\author{Jeffrey Boerner\\jboerner@math.uiowa.edu\\ \\Paul Drube\\pdrube@math.uiowa.edu}

\begin{document}

\maketitle

\abstract{Frobenius extensions play a central role in the link homology theories based upon the $sl(n)$ link variants, and each of these Frobenius extensions may be recast geometrically via a category of marked cobordisms in the manner of Bar-Natan.  Here we explore a large family of such marked cobordism categories that are relevant to generalized $sl(n)$ link homology theories.  We also investigate the skein modules that result from embedding these marked cobordisms within 3-manifolds, and arrive at an explicit presentation for several of these generalized skein modules.}

\section{Introduction}

In \cite{khovanov1}, Mikhail Khovanov developed a link homology theory categorifying the Jones polynomial.  That homology theory utilized a particular rank-2 Frobenius extension to define the boundary operator of its chain complex.  Utilizing the well-known correspondence between Frobenius extensions and 2-D Topological Quantum Field Theories (TQFTs), Dror Bar-Natan presented an equivalent development of Khovanov's homology that made use of a category of ``marked cobordisms" (\cite{barnatan}).  This marked cobordism category was taken modulo three sets of local relations that recast algebraic properties of Khovanov's Frobenius extension in geometric terms.  Khovanov proceeded to develop a second major link homology in \cite{Khovanov2}- an $sl(3)$ link homology that made use of a rank-3 Frobenius extension.  In \cite{mackaay}, Marco Mackaay and Pedro Vaz extended his result using a family of ``universal $sl(3)$ Frobenius extensions", and along the way mimicked Bar-Natan's marked cobordism category as part of their category of foams.  One intention of this paper is to generalize these marked cobordism constructions to Frobenius extensions of all ranks $n \geq 2$.

In \cite{asaeda}, Marta Asaeda and Charles Frohman introduced the notion of embedding Bar-Natan's marked cobordisms within a 3-manifold, thus producing skein modules. They defined and explored Bar-Natan's original skein module and gave explicit presentations of that skein module for several simple 3-manifolds.  Uwe Kaiser made these ideas rigorous in \cite{kaiser}, developing skein modules based on TQFTs from any Frobenius extension.  His work gives us a multitude of skein modules to investigate, and in this paper we closely examine a large family of Frobenius extensions that generalize those extensions associated to the $sl(n)$ link homology theories.  Our work requires us to prove a number of foundational results about Frobenius extensions that do not seem to appear anywhere else in the literature.  In particular, a significantly generalized version of Bar-Natan's original ``neck-cutting relation" is investigated, especially as it relates to the evaluation of closed manifolds in a skein module.

This paper is structured as follows:  In Section \ref{sec: towards skein modules} we review the underpinnings of Frobenius extensions and 2-D TQFTs, culminating in a detailed description of Bar-Natan's  original category of marked cobordisms.  Section \ref{sec: skein modules} generalizes this category to the class of Frobenius extensions in question, while Section \ref{sec: properties of skein modules} is concerned largely with neck-cutting in these extensions and concludes with a major theorem regarding the evaluation of closed compact surfaces.  In Section \ref{sec: embedded skein modules} we finally arrive at skein modules, proving several general facts about the entire family of skein modules (with special emphasis on rank $n=2$ Frobenius extensions).  We also thoroughly compute an example, in part to demonstrate how complex these skein modules become without certain simplifying assumptions about the underlying Frobenius extensions.  The appendix tackles some of the difficult computational challenges revealed in Sections \ref{sec: skein modules} and \ref{sec: properties of skein modules} via linear algebra.  Although intended largely as a curiosity, this appendix is interesting in that it betrays a deep indebtedness to the theory of symmetric polynomials. 

\section{Towards Skein Modules: Frobenius Extensions \& 2-D TQFTs}
\label{sec: towards skein modules}

\subsection{Frobenius Extensions}
\label{subsec: frobenius extensions}

As in \cite{Khovanov3}, we begin with a ring extension $\iota: R \hookrightarrow A$ of commutative rings with $1$ such that $\iota(1)=1$.  $\iota$ endows $A$ with the structure of a $R$-bimodule, allowing for an obvious restriction functor $R: \amod \rightarrow \rmod$.  By definition, $\iota$ is a Frobenius extension if this functor $R$ has a two-sided adjoint.  More specifically, $\iota$ is Frobenius if the induction functor $T: M_R \mapsto (M \otimes_R A)_A$ and the coinduction functor $H: M_R \mapsto (Hom_R(A,M))_A$ are isomorphic as functors $T,H: \rmod \rightarrow \amod$.  For the remainder of this paper we will treat Frobenius extensions such that $A$ is finitely-generated and projective as an $R$-module.

The functor isomorphism above prompts A-linear isomorphisms $End_R(A) \cong A \otimes_R A$ and $A^* \cong A$ (corresponding, respectively, to $M=A$ and $M=R$).  In the finite projective case, the latter of those A-linear isomorphisms leads to the alternative definition of a Frobenius extension as a ring extension such that $A$ is self-dual as a $R$-module.  Equivalently, a finite projective Frobenius extension is a ring extension such that $A$ is equipped with an $R$-linear comultiplication map $\Delta : A \rightarrow A \otimes_R A$ that is coassociative, cocommutative, and in possession of an $R$-linear counit map $\varepsilon: A \rightarrow R$.  Stated below is yet another equivalent formulation of Frobenius extension that will be the primary definition utilized here.  For a detailed proof of the equivalence between these definitions see Kadison \cite{Kadison}, where they deal with the more general case of $R,A$ not necessarily commutative.

\begin{definition}
\label{def: Frobenius extension}
A {\bf Frobenius extension} is a finite projective ring extension $\iota: R \hookrightarrow A$ of commutative rings such that there exists a non-degenerate $R$-linear map $\varepsilon: A \rightarrow R$ and a collection of tuples $(x_i,y_i) \in A \times A$ such that, for all $a \in A$, $a = \displaystyle{\sum_i} x_i \varepsilon(y_i a) = \displaystyle{\sum_i} \varepsilon(a x_i) y_i$.
\end{definition}

$\varepsilon$ is known as the Frobenius map, or trace, and is identified with the counit map mentioned above.  For a Frobenius form to be nondegenerate means that there are no (principal) ideals in the nullspace of $\varepsilon$.  $(x_i,y_i)$ is referred to as the dual basis, with the duality condition taking the form $\varepsilon(x_i y_j) = \delta_{i,j}$.  $(R,A,\varepsilon, (x_i,y_i))$ is collectively referred to as the Frobenius system, and offers a complete description of the Frobenius extension.

One important feature of Frobenius extensions is that the aforementioned self-duality $A \cong A^*$ prompts an R-module isomorphism $A \otimes A \cong A \otimes A^* \cong End(A)$ that is given by $a \otimes b \mapsto a \varepsilon(b \underline{\ \ })$.  This map actually extends to an isomorphism of the underlying rings, as long as one defines a multiplication on $A \otimes A$ (known as the $\varepsilon$-multiplication) by $(a \otimes b)(a' \otimes b')=a \varepsilon(ba') \otimes b' = a \otimes \varepsilon(ba')b'$.  Note that, via the fundamental property of the Frobenius form in Definition \ref{def: Frobenius extension}, $\sum (x_i \otimes y_i)$ serves as the unit in this multiplication.  For full details of this construction, see \cite{Kadison}.

For a fixed base ring $R$, the fundamental notion of equivalence between two Frobenius systems $(R,A,\varepsilon, (x_i,y_i))$ and $(R,\tilde{A},\tilde{\varepsilon}, (\tilde{x}_i,\tilde{y}_i))$ is known as Frobenius isomorphism.  A Frobenius isomorphism is any $R$-linear ring isomorphism $\phi: A \rightarrow \tilde{A}$ such that $\varepsilon = \tilde{\varepsilon} \phi$, with the latter condition implying an isomorphism between the comodule structures of $A$ and $\tilde{A}$.  When we also have that $A = \tilde{A}$, Kadison alternatively states that two systems are equivalent iff $\varepsilon = \tilde{\varepsilon}$, which occurs iff $\sum (x_i \otimes y_i) = \sum (\tilde{x}_i \otimes \tilde{y}_i)$.  The second iff above follows from the fact that both sums serve as the unit element for the $\varepsilon$-multiplication.  This final observation will prove especially significant in Section \ref{sec: properties of skein modules}.  In Section \ref{sec: properties of skein modules}, we'll also see how these two distinct notions of equivalence coincide when $A = \tilde{A}$.

\subsection{2-D TQFTs}
\label{subsec: TQFTs and Skein Modules}

There is a well-known correspondence between Frobenius extensions $R \hookrightarrow A$ and 2-dimensional TQFTs over $R$.  See \cite{atiyah} or \cite{Kock} for a very detailed discussion of this correspondence in the less general setting of ``Frobenius algebras"- where $R$ is a field.  Very briefly, a 2-dimensional TQFT $Z$ is a symmetric monoidal functor from the category of oriented 2-dimensional cobordisms $\cobtwo$ to the category of (left) $R$-modules $\rmod$.  On the object level, $Z$ sends $S^1$ to an $R$-module $A$ that necessarily acts as a Frobenius extension of $R$.  Being a symmetric monoidal functor, $Z$ then sends $n$ disjoint copies of $S^1$ to the tensor product $A^{\otimes_n}$ and the empty 1-manifold to the base ring $R$.

On the morphism level, $Z$ sends a 2-dimensional cobordism $N$ between $X$ and $X'$ to an $R$-linear map $Z(N):Z(X) \rightarrow Z(X')$.  In the case of a closed 2-manifold we have $X=X'=\emptyset$ and hence that $Z(N) \in End_R(R)$.  We identify this map with an element of $R$ via the image of $1 \in R$, meaning that $Z$ determines an $R$-valued invariant of closed 2-manifolds.  This final property will be a central component in the upcoming discussion.

In recent years, Frobenius extensions have seen widespread usage in the construction of link invariants.  This dates back to Khovanov's work in \cite{khovanov1}, where he utilized a Frobenius extension to construct the boundary operator in his homology theory that categorified the Jones polynomial.  The Frobenius extension used by Khovanov was $R=\Z$, $A=\Z[x]/(x^2)$, with Frobenius form defined on the $R$-module basis $\lbrace 1,x \rbrace$ by $\varepsilon(1)=0,\varepsilon(x)=1$.  The resulting dual-basis for this extension was then $\lbrace (1,x),(x,1)\rbrace$.

In \cite{barnatan}, Dror Ban-Natan utilized the correspondence between 2-D TQFTs and Frobenius extensions to give a more geometric interpretation of Khovanov's homology.  His primary construction was a category of ``decorated cobordisms", which utilized the associated TQFT to represent algebraic properties of the Frobenius extension via 2-dimensional surfaces.  This category $\cobtwo_A$ has the same objects as $\cobtwo$, but its morphisms (2-D cobordisms) may now by ``marked" by elements of $A$ (where a ``marking" appears as an element of $A$ written on the desired component, and an ``unmarked" surface corresponds to a marking by $1 \in A$).  Markings are allowed to ``move around" and to be multiplied together (or factored) within a fixed component, but are not allowed to ``jump" across to a distinct components of the same cobordism.  The category is also taken to be $R$-linear, and as a convention we always write elements of $R \subseteq A$ ``in front" of cobordisms to emphasize this linear structure.

The morphisms of $\cobtwo_A$ are then taken modulo three sets of local relations $l$ that actually encapsulate the algebraic information about $R \hookrightarrow A$, and it is the resulting quotient category $\cobtwo_A / l$ that is actually of interest.  As we will soon be generalizing this construction to more general Frobenius extensions, these local relations are described in detail below.  For the original description of these relations, see \cite{barnatan}\\
\\
\underline{Sphere Relations}\\
In the previous subsection we mentioned how $\varepsilon:A \rightarrow R$ is identified with the counit map.  In $\cobtwo$ this counit takes the form of the ``cap" surface.  Using marked surfaces allows us to graphically depict the value of $\varepsilon$ at a specific value $a \in A$ via precomposition of this ``cap" with a ``cup" decorated by $a$ (this ``cup" is nothing more than the unit map $u:R \rightarrow A$, so we are actually interpreting a marked ``cup" as $a=a*u(1)$).

This all gives rise to Bar Natan's ``sphere relations", which correspond to the function values $\varepsilon(1) = 0$ and $\varepsilon(x) = 1$.  They say that we may remove a (disjoint) unmarked sphere from a cobordism and multiply the entire cobordism by $0$, or remove a sphere marked by $x$ and multiply the entire cobordism by $1$.  Note that, merely for this sub-section, we adopt Bar Natan's original notation of a dot corresponding to a surface marked by $x$.

\[
\begin{picture}(40,30)
\raisebox{30pt}{\scalebox{.15}{\includegraphics[angle=270]{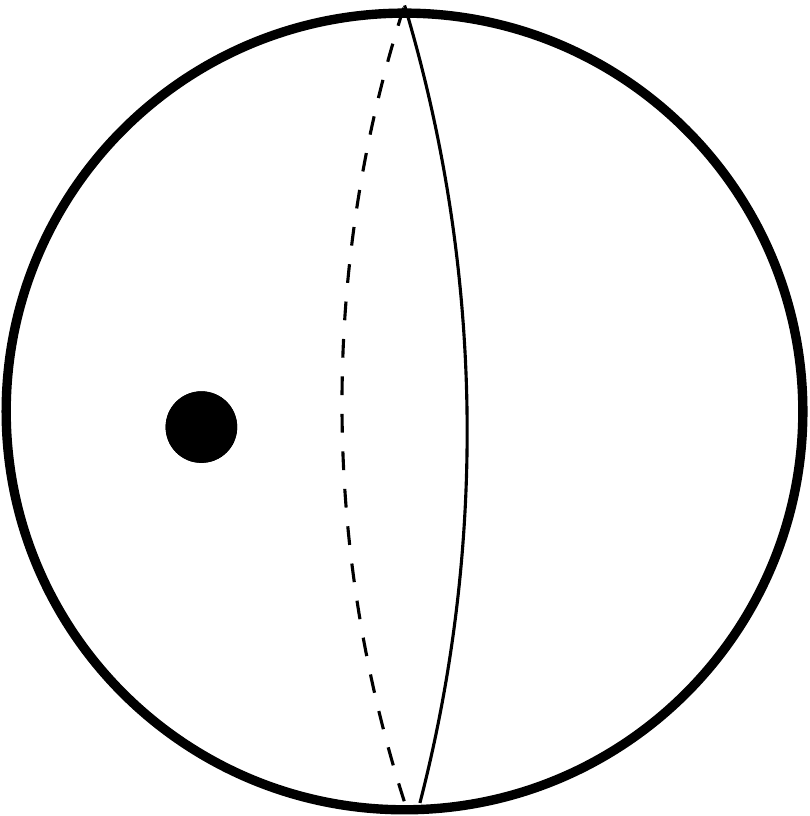}}}
\end{picture}
\raisebox{10pt}{\ \ = \ \ 1 \ \ \ \ \ \ } 
\begin{picture}(40,30)
\raisebox{30pt}{\scalebox{.15}{\includegraphics[angle=270]{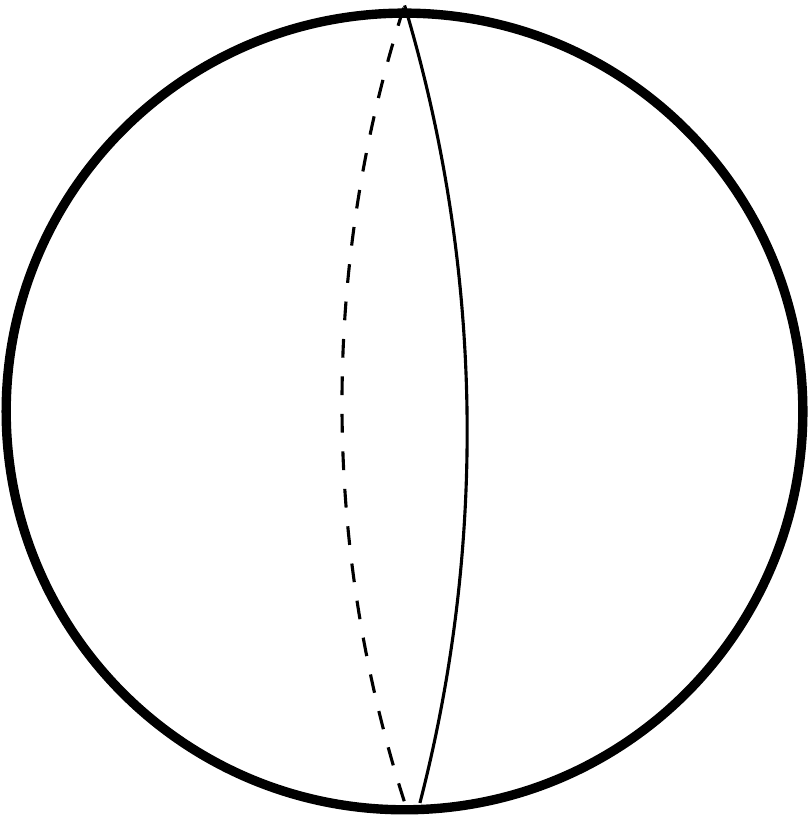}}}
\put(-20,18){\fontsize{9}{10.8}$1$}
\end{picture}
\raisebox{10pt}{\ \ = \ \ 0}
\]
\\
\underline{``Dot Reduction'' Relation}\\
The ``dot reduction" relation follows directly to our choice of $A$, and allows us to re-decorate surfaces by equivalent elements in $A$.  As the $R$-module $A$ has a single generating relation in $p(x)=x^2=0$, all such local relations are generated by the one shown below.

\[
\begin{picture}(40,30)
\raisebox{30pt}{\scalebox{.12}{\includegraphics[angle=270]{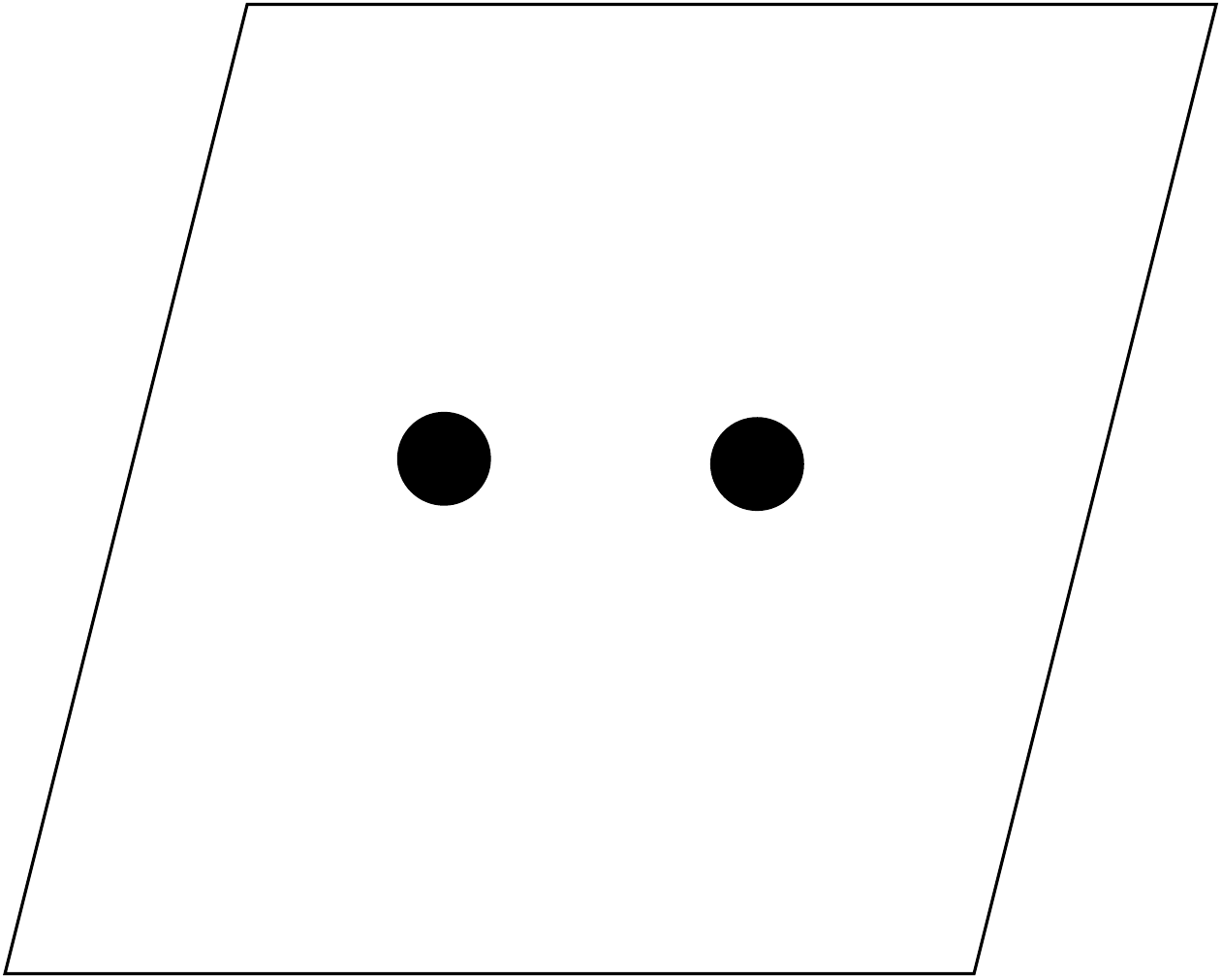}}}
\end{picture}
\raisebox{10pt}{\ \ = \ \ 0}
\]
\\
\underline{Neck-Cutting Relation}\\
The last relation can be applied to any surface $N$ with a compression disk- a copy of $D^2$ such that $\partial D^2 \subset N$ and the interior of $D^2$ is disjoint from $N$.  It allows one to ``neck-cut" along the compression disk, and then replace $N$ with a sum of surfaces in which a regular neighborhood of $D^2 \cap N$ has been removed and replaced by two copies of $D^2$ along the two new boundary components.  For Bar Natan's category $\cobtwo_A$, this ``neck-cutting" relation is depicted below.

Algebraically, this relation follows from our choice of dual-basis $\lbrace(1,x),(x,1)\rbrace$.  In particular, it is a result of the dual basis' non-degeneracy condition that $a = \sum_{i=1}^{n} \varepsilon (a x_i) y_i$ for all $a \in A$, which we interpret as an equality between two endomorphisms of $A$.  The left-hand side of the equation sends $a \in A$ to $a$ and hence is identified with the identity map $1_A: A \rightarrow A$: the image of a straight cylinder via our TQFT.  The endomorphism on the right applies $\varepsilon (x_i \underline{\ \ })$ to the input (corresponding to the ``cap" decorated by $x_i$) and then outputs $y_i = y_i u(1)$ (corresponding to the ``cup" decorated by $y_i$).

\[
\raisebox{29pt}{\scalebox{.2}{\includegraphics[angle=270]{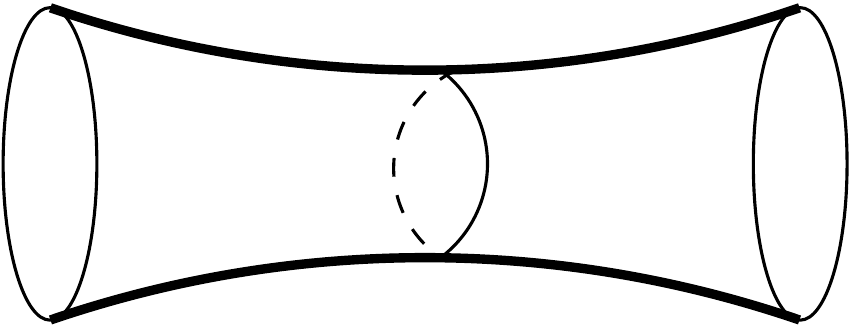}}}
\ \ = \ \ 
\raisebox{29pt}{\scalebox{.2}{\includegraphics[angle=270]{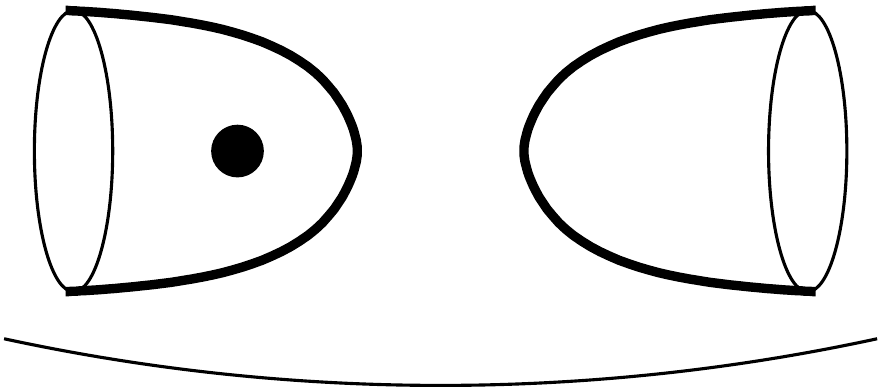}}}
\ + \ 
\raisebox{29pt}{\scalebox{.2}{\includegraphics[angle=270]{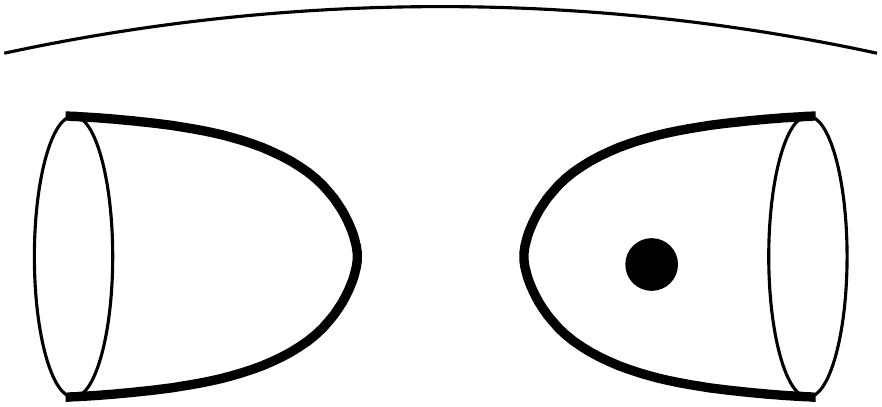}}}
\]
\\
In practice, the 2-D surfaces of interest will be closed.  If we view such surfaces abstractly (i.e.- not embedded within some higher-dimensional space), we can use the local relations above to compress all surfaces down to $R$-linear combinations of spheres that may then be evaluated to a constant in $R$.  The nice thing is that, since our local relations were determined by our Frobenius extension, when we restrict ourselves to unmarked surfaces these constants are identical to those of the $R$-valued 2-manifold invariant of the associated 2-D TQFT.

The concerns of this paper will be more general in several respects.  Most fundamentally, we will consider skein modules, in which our marked cobordisms of $\cobtwo_A/l$  will actually be embedded within a 3-manifold.  This situation was first examined by Frohman and Asaeda in \cite{asaeda}, and complicates the theory in that the topology of the chosen 3-manifold may prevent us from compressing all surfaces down to disjoint spheres.  Such embedded skein modules will be the primary focus of Section \ref{sec: embedded skein modules}.

Before considering such skein modules, it will be our primary goal to generalize the category $\cobtwo_A/l$ presented above by considering more general Frobenius extensions $R \hookrightarrow A$.  When $A$ is a rank 2 $R$-module, a number of different Frobenius extensions have been used to produce interesting link homologies (\cite{khovanov1},\cite{lee}), and in \cite{Khovanov3} Khovanov has even described a ``universal" rank-2 Frobenius extension that nicely encapsulates all of the distinct theories (although none of this is done within the framework of cobordisms).  We are more strongly motivated by the work Mackaay and Vaz in their construction of a universal ``$sl(3)$-link homology" (\cite{mackaay}), in which they introduced a set of local relations similar to those above as part of their category \cat{Foam}.

The ``universal $sl(3)$ Frobenius extension" utilized by Mackaay and Vaz was $R=\C$, $A=\C[x]/(x^3-a x^2 - b x -c)$, with Frobenius form defined on the basis $\lbrace 1,x,x^2 \rbrace$ by $\varepsilon(1)=\varepsilon(x)=0,\varepsilon(x^2)=1$.  Notice that this extension is not completely general, in the sense that it doesn't incorporate all possible rank-3 extensions of $\C$ (only ones with a cyclic basis and with a very specific type of Frobenius form).  Coming up with a truly ``universal" extension is only a tractable problem in the rank-2 case, where there is always a presentation of the rank-2 $R$-module $A$ of the form $A=R[x]/p(x)$ for some quadratic $p(x) \in R[x]$.  Thus in generalizing to arbitrarily high ranks $n \geq 2$, we chose to restrict our attention to a specific family of Frobenius extensions that most closely mimic this ``universal $sl(3)$ extension".  These ``$sl(n)$ Frobenius extensions" are the subject of Sections \ref{sec: skein modules} and \ref{sec: properties of skein modules}.

\section{$sl(n)$ Frobenius extensions}
\label{sec: skein modules}

The rank-n Frobenius extensions $R \hookrightarrow A$ that we consider in this paper will be of the form $A=R[x]/p(x)$, where $p(x) \in R[x]$ is a (monic) degree-$n$ polynomial, and we have Frobenius form defined on the standard basis $\lbrace 1,x,...,x^{n-1} \rbrace$ by $\varepsilon(1)=...\varepsilon(x^{n-2})=0,\varepsilon(x^{n-1})=1$.  To standardize notation, let $p(x)=x^n - a_1 x^{n-1} - ... - a_n$.  From now on we refer to these systems as universal $sl(n)$ Frobenius extensions.

Notice that, in the case of $n=3$ and $R=\C$, this extension coincides with the universal $sl(3)$ extension of Mackaay and Vaz.  As opposed to the ``easy choice" of $R=\C$, in the spirit of Khovanov we will choose to work with $R=\Z[a_1,...,a_n]$, a more general setting that will make several of our proofs slightly more involved.  Always working with the standard basis $\lbrace 1,x,...,x^{n-1} \rbrace$, we endow our rank-n system with a standard dual-basis as below:

\begin{lemma}
\label{thm: dual basis}
The rank-n Frobenius system defined above has a dual basis given by:

\vspace{-.2 in}

\begin{equation*}
 \{ (x^{n-1}, 1), (x^{n-2}, x - a_1), (x^{n-3}, x^2 - a_1 x - a_2), \ldots, (1, x^{n-1} - a_1 x^{n-2} - \ldots - a_{n-2} x - a_{n-1}) \}.
 \end{equation*}
 
\end{lemma}
\begin{proof}
This follows immediately from the fact that the inverse of $\lambda = [[\varepsilon(x^{i+j-2})]]$, after reducing $mod(p(x))$, is of the form
\[
\lambda^{-1}=
\begin{bmatrix}
-a_{n-1} & -a_{n-2} & \ldots & -a_1 & 1 \\
-a_{n-2} & \ldots & \ldots & 1 & 0 \\
\vdots & \vdots & \vdots & \vdots & \vdots \\
-a_1 & 1 & 0 & \ldots & 0 \\
1 & 0 & \ldots & \ldots & 0
\end{bmatrix}
\]
\end{proof}

Notice that for the universal rank-2 case this gives the familiar dual basis of $\lbrace(x,1),(1,x-a_1)\rbrace$, while for the universal rank-3 case we have $\lbrace(x^2,1),(x,x-a_1),(1,x^2-a_1x-a_2)\rbrace$.

Equipped with a dual basis, we are now ready to determine the local relations for our universal rank-n Frobenius extension.  Generalizing the presentation from Subsection \ref{subsec: TQFTs and Skein Modules}, we divide these local relations into the three groups that are outlined in detail below.  It is in this general rank $n \geq 2$ case that our distaste for Bar-Natan's ``dot notation" is finally justified, as it obviously becomes unwiedly with increasing rank.\\
\\
\underline{Sphere Relations}\\
Here our Frobenius form is defined by $\varepsilon(x^{n-1})=1$ and $\varepsilon(x^k)=0$ for $0 \leq k \leq n-2$.  It follows that the sphere relations for the universal rank-n system are:

\[
\begin{picture}(40,30)
\raisebox{30pt}{\scalebox{.15}{\includegraphics[angle=270]{sphere.pdf}}}
\put(-25,17){\fontsize{9}{10.8}$x^{n\kern-1pt{-}\kern-1pt{1}}$}
\end{picture}
\raisebox{10pt}{\ \ = \ \ 1 \ \ \ \ \ \ } 
\begin{picture}(40,30)
\raisebox{30pt}{\scalebox{.15}{\includegraphics[angle=270]{sphere.pdf}}}
\put(-20,18){\fontsize{9}{10.8}$1$}
\end{picture}
\raisebox{10pt}{\ \ = \ \ \ldots \ \ = \ \ } 
\begin{picture}(40,30)
\raisebox{30pt}{\scalebox{.15}{\includegraphics[angle=270]{sphere.pdf}}}
\put(-25,17){\fontsize{9}{10.8}$x^{n\kern-1pt{-}\kern-1pt{2}}$}
\end{picture}
\raisebox{10pt}{\ \ = \ \ 0}
\]
\\
\underline{``Dot Reduction'' Relation}\\
For $A=R[x]/p(x)$, where $p(x) = x^n - a_1 x^{n-1} - ... - a_n$, we have $x^n = a_1 x^{n-1} + ... + a_n$ and the ``dot reduction" relation allows us to re-mark any fixed component of our surface as:

\[
\begin{picture}(45,30)
\raisebox{20pt}{\scalebox{.05}{\includegraphics[angle=270]{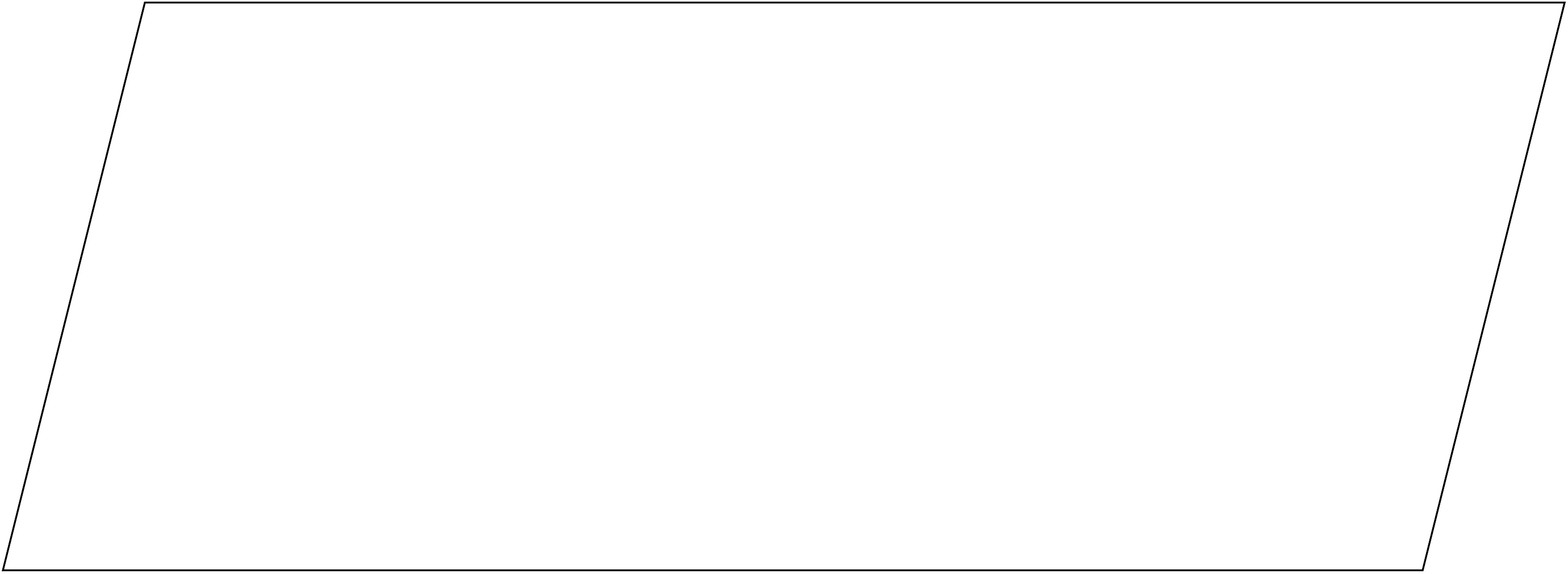}}}
\put(-25,10){\fontsize{9}{10.8}$x^n$}
\end{picture}
\raisebox{10pt}{\ \ = \ }
\begin{picture}(80,30)
\raisebox{30pt}{\scalebox{.12}{\includegraphics[angle=270]{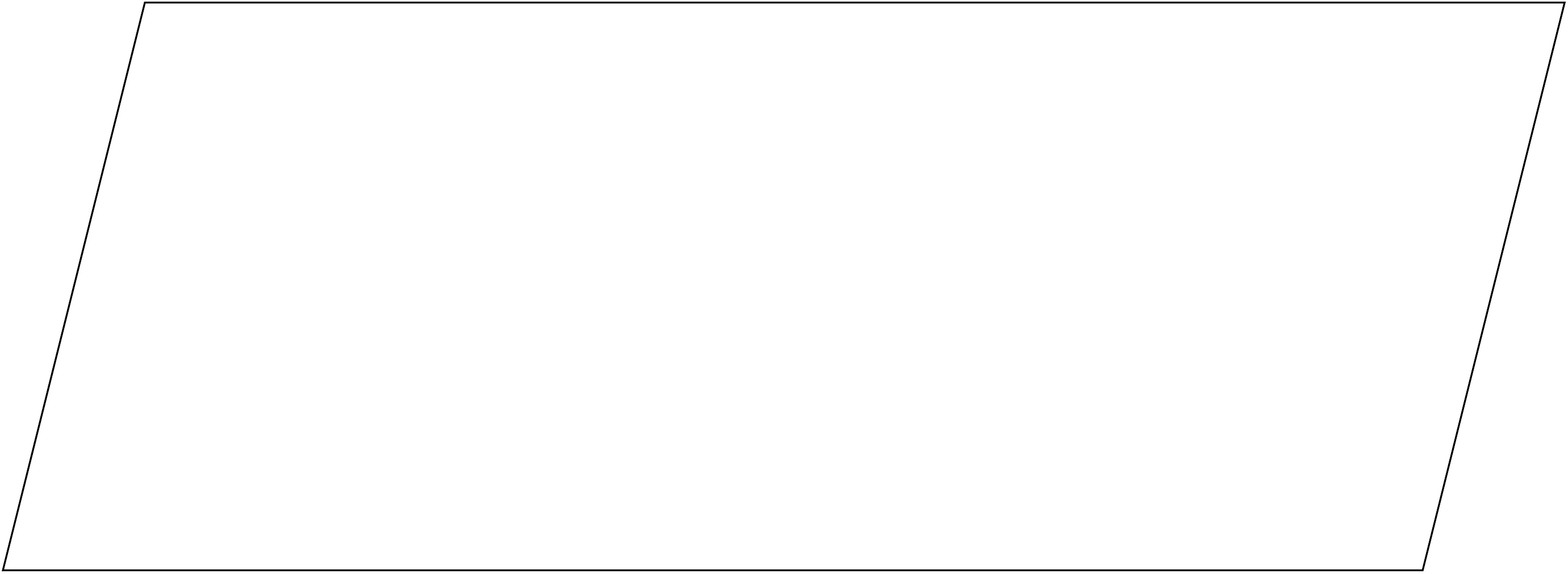}}}
\put(-80,10){\fontsize{9}{10.8}$a_1 x^{n-1} + ... + a_n$}
\end{picture}
\]
\\
\underline{Neck-Cutting Relation}\\
In Lemma \ref{thm: dual basis} we saw that the dual basis for the universal rank-n Frobenius system was $\lbrace (x^{n-1}, 1), (x^{n-2}, x - a_1), (x^{n-3}, x^2 - a_1 x - a_2), \ldots, (1, x^{n-1} - a_1 x^{n-2} - \ldots - a_{n-2} x - a_{n-1}) \rbrace$.  Grouping terms via the coefficients $a_i$, the corresponding neck-cutting relation then takes the elegant form:

\[
\raisebox{29pt}{\scalebox{.2}{\includegraphics[angle=270]{leftside.pdf}}}
\ \ = \ \ 
\begin{picture}(68,0)
\raisebox{29pt}{\scalebox{.2}{\includegraphics[angle=270]{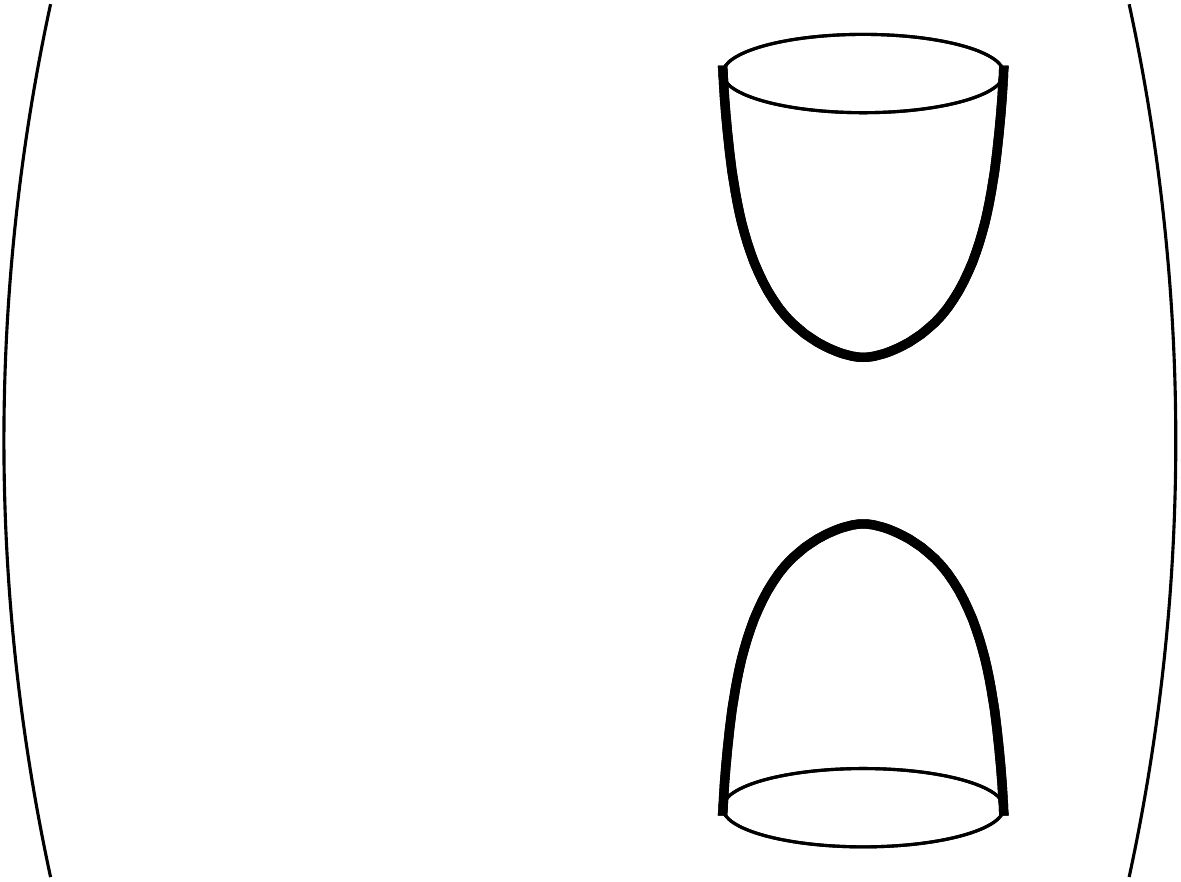}}}
\put(-57,0){\fontsize{16}{19.2}$\sum$}
\put(-58,-12){\fontsize{7}{8.4}$i\kern-.5pt{+}\kern-.5pt{j}=$}
\put(-56,-20){\fontsize{7}{8.4}$n\kern-.5pt{-}\kern-.5pt{1}$}
\put(-22,14){\fontsize{9}{10.8}$x^i$}
\put(-22,-14){\fontsize{9}{10.8}$x^j$}
\end{picture}
\ - \ a_1 \
\begin{picture}(68,0)
\raisebox{29pt}{\scalebox{.2}{\includegraphics[angle=270]{ndots.pdf}}}
\put(-57,0){\fontsize{16}{19.2}$\sum$}
\put(-58,-12){\fontsize{7}{8.4}$i\kern-.5pt{+}\kern-.5pt{j}=$}
\put(-56,-20){\fontsize{7}{8.4}$n\kern-.5pt{-}\kern-.5pt{2}$}
\put(-22,14){\fontsize{9}{10.8}$x^i$}
\put(-22,-14){\fontsize{9}{10.8}$x^j$}
\end{picture}
\ - \ \ldots \ - \ a_{n-1} \
\raisebox{29pt}{\scalebox{.2}{\includegraphics[angle=270]{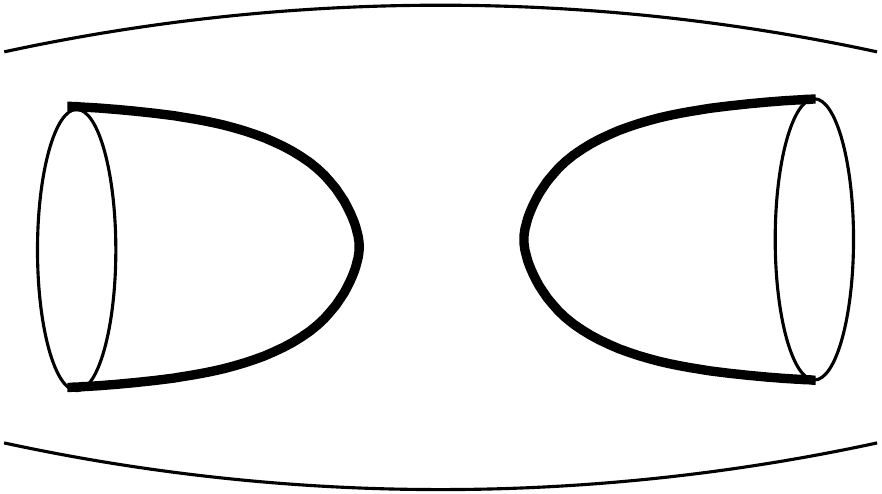}}}
\]

\section{Properties of $sl(n)$ Frobenius Extensions}
\label{sec: properties of skein modules}

\subsection{Neck-Cutting}
\label{subsec: neck-cutting}

We concluded Section \ref{sec: skein modules} by presenting a neck-cutting relation that was associated with the universal $sl(n)$ Frobenius extension.  Here we consider what that relation tells us when the curve that bounds our compression disk is non-separating: when the ``top" and ``bottom" surfaces from the neck-cutting equation are now on the same component.  By $R$-linearity and our aforementioned ability to multiply distinct decorations upon a fixed component, in this case neck-cutting amounts to multiplication on the effected component by a ``genus reduction" term of $g = \sum_{i=1}^n x_i y_i \in A$.  This value always coincides with $m \circ \Delta (1)=m(\sum_{i=1}^n x_i \otimes y_i)=\sum_{i=1}^n x_i y_i$ via the definition of comultiplication and multiplication in any Frobenius extension, a correspondence that is illustrated in Figure 1 below.  Also note that this $g$ is the same as what Kadison and others define to be the $\varepsilon$-index of the ring extension $R \hookrightarrow A$.  We choose our slightly unorthodox notion in order to emphasize its geometric importance within the category of marked cobordisms.

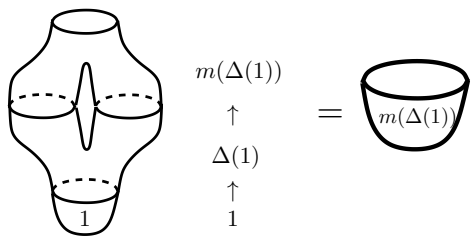
\begin{figure}[htb]
\label{decompfig}
\caption{A handle.}
\begin{center}
\scalebox{.4}{\includegraphics[height = 3 in]{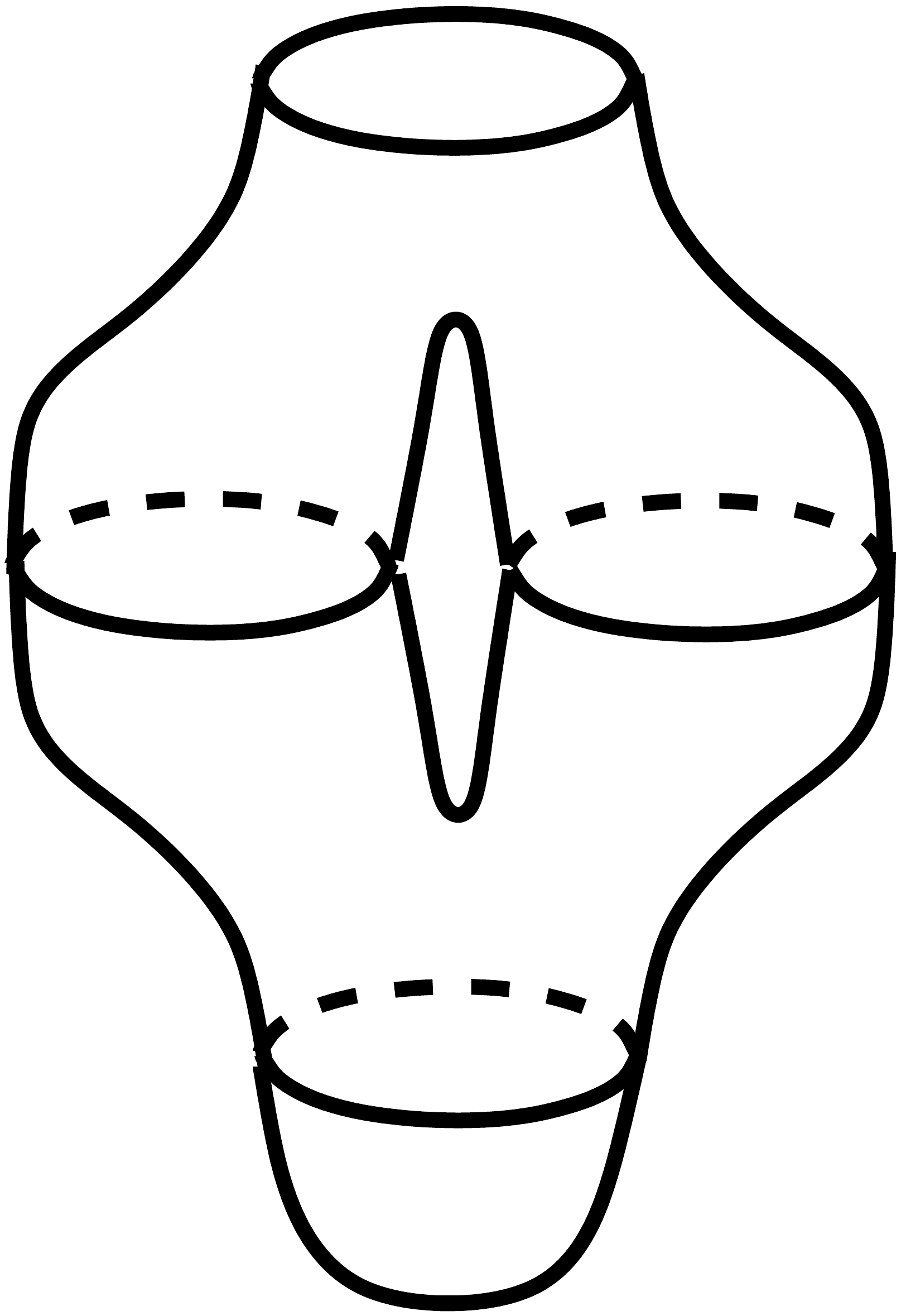}\put(-80,10){\scalebox{2}{$1$}}\put(60,10){\scalebox{2}{$1$}}\put(60,38){\scalebox{2}{$\uparrow$}}\put(45,70){\scalebox{2}{$\Delta(1)$}}\put(60,110){\scalebox{2}{$\uparrow$}}\put(30,150){\scalebox{2}{$m(\Delta(1))$}}
\hspace{2 in}\raisebox{1.5 in}{\scalebox{3}{$=$}\hspace{.2 in}\raisebox{-.4 in}{\includegraphics[height = 1.25 in]{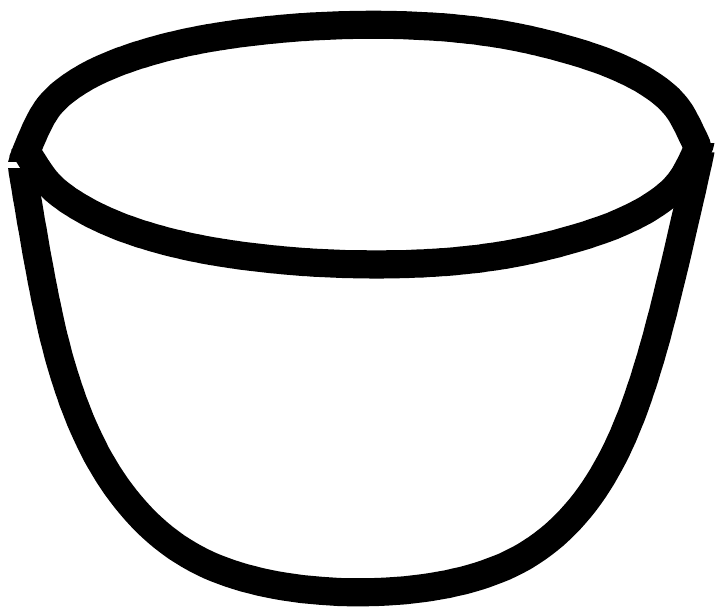}}}\put(-88,108){\scalebox{1.8}{$m(\Delta(1))$}}}
\end{center}
\end{figure}

In particular, $T^2$ decorated with $1$ is equivalent to $S^2$ decorated with $g$.  It follows that the Frobenius form evaluates an unmarked torus to $\varepsilon(g) = n$ in the universal $sl(n)$ Frobenius extension, which is compatible with the fundamental result that, in any 2-D TQFT $Z$, $Z(T^2) \in \Z[a_1,...,a_n]$ equals the rank of the associated Frobenius extension.  Similarly, a genus-$i$ closed, compact surface $\Sigma_i$ decorated with $1$ is equivalent to $S^2$ decorated with $g^i$, as shown in Figure 2.  It follows that an unmarked $\Sigma_i$ is evaluated by our TQFT $Z$ as $\varepsilon(g^i)$.  Note that in this higher genus situation there is some ambiguity in how we choose our compression disks, and to achieve the succinct result above we need to ensure that the $i$ curves bounding those disks are all non-separating (although, naturally, any two ways of compressing down to a incompressible surface must evaluate similarly via $Z$!).

As suggested by its alternative title of $\epsilon$-index, $g$ is dependent upon not only the rings $R,A$ but also upon the choice of Frobenius form $\varepsilon$.  What follows are a series of lemmas that hope to characterize how $g$ behaves under changes in Frobenius structure.

\begin{figure}[htb]
\label{cuttingtosphere}
\caption{Cutting down to a sphere.}
\begin{center}

\scalebox{.5}{\raisebox{-1 in}{\includegraphics[height = 1.5 in]{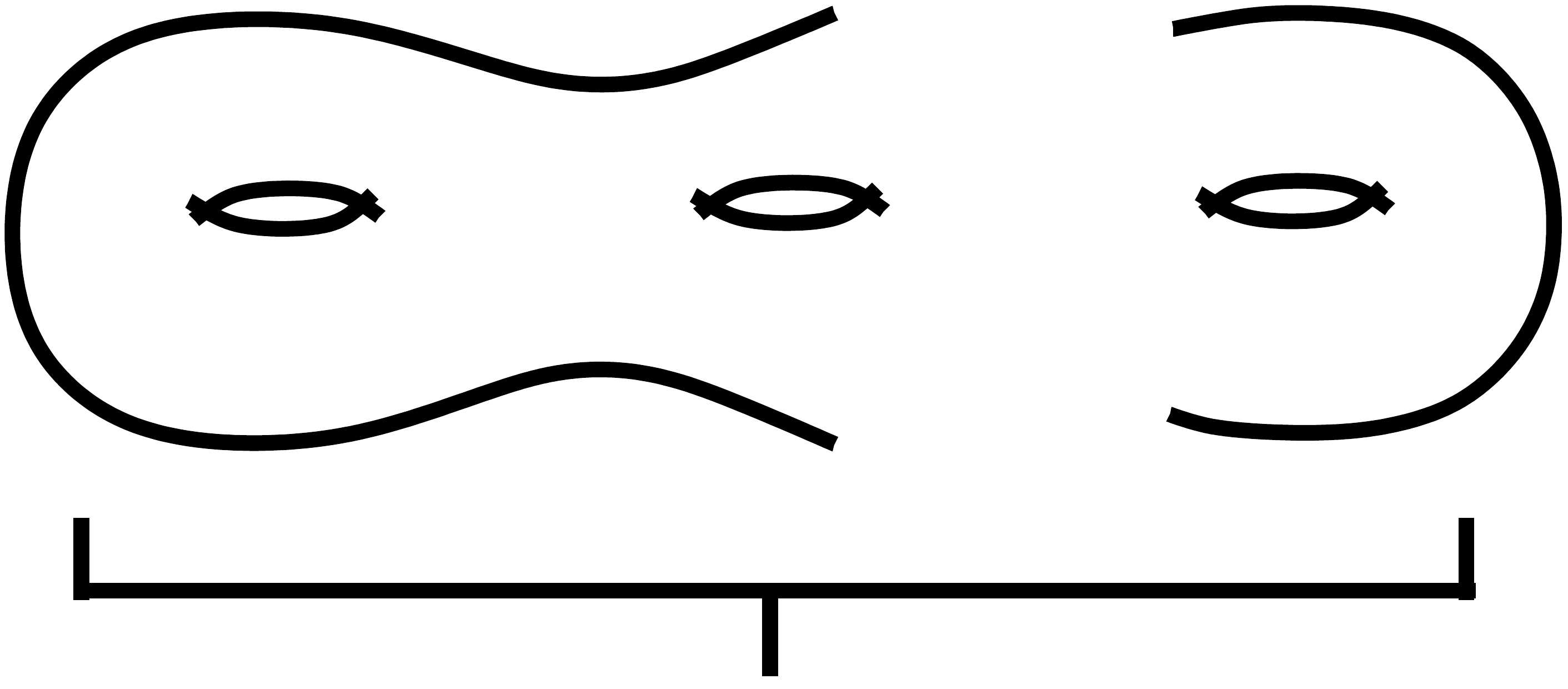}\put(-100,105){\scalebox{2}{$\dots$}}\put(-100,45){\scalebox{2}{$\dots$}}\put(-128,-15){\scalebox{2}{$i$}}}\put(-163,-2){\scalebox{2}{$1$}} \scalebox{2}{ $=$ } \raisebox{-.5 in}{\includegraphics[height = 1 in]{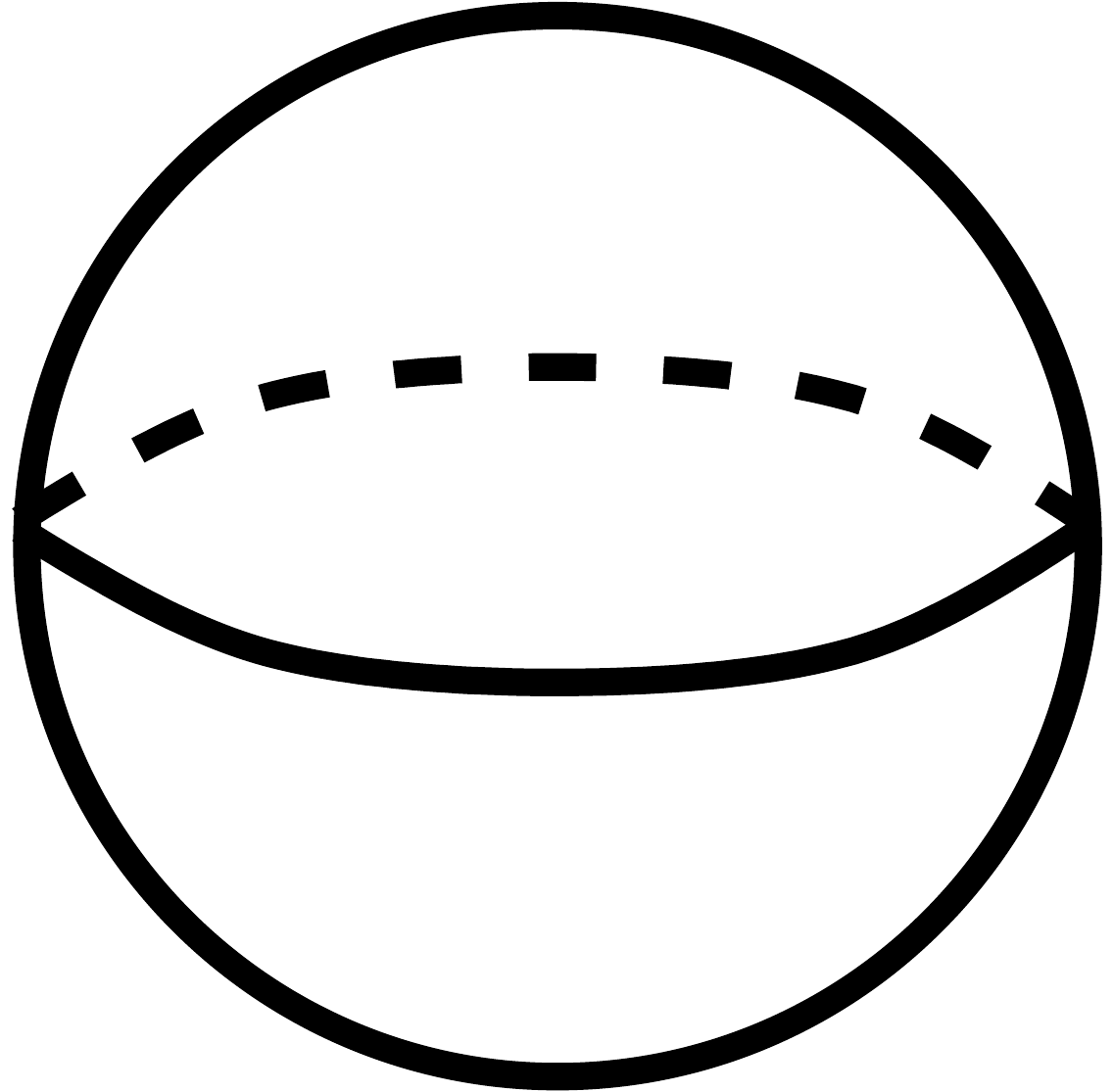}\put(-40,9){\scalebox{2}{$g^i$}}}}

\end{center}

\end{figure}
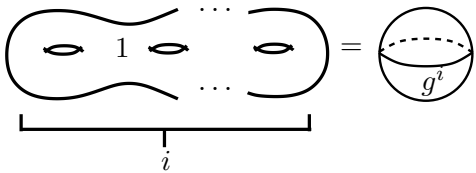

We begin by considering the case where our two Frobenius systems $(R,A,\varepsilon, (x_i,y_i))$ and $(R,A,\varepsilon, (\tilde{x}_i,\tilde{y}_i)$ differ only via the choice of dual basis.  In Subsection \ref{subsec: frobenius extensions} we noted that $\varepsilon = \tilde{\varepsilon}$ iff $\sum (x_i \otimes y_i) = \sum (\tilde{x}_i \otimes \tilde{y}_i) \in A \otimes A$.  Given the latter equality, the well-definedness of the $A$-linear map $A \otimes A \rightarrow A$, $a \otimes b \mapsto ab$ ensures that $g = \sum x_i y_i = \sum \tilde{x}_i \tilde{y}_i = \tilde{g}$.  Hence $g$ is independent of our choice of dual basis.

The following well-known result characterizes all possible Frobenius forms over a fixed ring extension $R \hookrightarrow A$.  Its proof is readily available in \cite{Kock} or \cite{Kadison} (although \cite{Kock} presents an equivalent argument for the special case of Frobenius algebras with $R=k$ a field):

\begin{lemma}
\label{thm: changing frobenius systems}
Given two Frobenius systems over the same ring extension, $(R,A,\varepsilon, (x_i,y_i))$ and $(R,A,\tilde{\varepsilon}, (\tilde{x}_i,\tilde{y}_i)$, up to change of dual basis we have $(R,A,\tilde{\varepsilon}, (\tilde{x}_i,\tilde{y}_i)=(R,A,\varepsilon(d \underline{\ \ }), (x_i,d^{-1} y_i)$ for some invertible $d \in A$.  Moreover, there is a bijection between equivalence classes of Frobenius systems over $R \hookrightarrow A$ and invertible $d \in A$.
\end{lemma}

Pause to note that, in light of this lemma, it becomes clear that the ``standard" notion of Frobenius equivalence coincides with Kadison's definition when $R,A$ are fixed.  Specifically, the $R$-linear ring automorphism $\phi: A \rightarrow A$ underlying any Frobenius equivalence forces $d=1$ above, thus ensuring $\varepsilon = \tilde{\varepsilon}$.  More germane to our discussion is the following corollary:

\begin{corollary}
\label{thm: epsilon dependence of g}
Given two Frobenius systems on the same ring extension, $(R,A,\varepsilon, (x_i,y_i))$ and $(R,A,\tilde{\varepsilon}, (\tilde{x}_i,\tilde{y}_i)$, their respective genus-reduction terms are related by $\tilde{g}=d^{-1} g$ for some invertible $d \in A$.
\end{corollary}
\begin{proof}
By Lemma \ref{thm: changing frobenius systems}, after an appropriate change of dual basis the second system is of the form $(R,A,\varepsilon(d \_ \_), (x_i,d^{-1} y_i)$ for some invertible $d \in A$.  This change of basis leaves $\tilde{g}$ unaffected, so by comparing this modified system to the first system from our theorem we have $\tilde{g}=\sum \tilde{x}_i \tilde{y}_i = \sum x_i d^{-1} y_i = d^{-1} g$.
\end{proof}

This argument obviously extends to show that $\tilde{g}^i = (d^{-1})^i g^i$ for all powers $i \geq 1$.  In general, this corollary tells us little about how $\tilde{g}$ or the $\tilde{g}^i$ are actually evaluated by the $R$-linear Frobenius form, as the necessary $d \in A$ need not be in $R$ and hence can't necessarily be ``pulled out" of the argument for $\varepsilon$ and $\tilde{\varepsilon}$.  However, do note that if $g^i = 0$ (or $g^i \neq 0$) for any Frobenius system over $R \hookrightarrow A$, then $\tilde{g}^i = 0$ ($\tilde{g}^i \neq 0$) for any Frobenius system over $R \hookrightarrow A$.

One final lemma that we will need is that a general Frobenius equivalence respects genus-reduction terms.  When $A = \tilde{A}$, this result follows directly from Lemma \ref{thm: changing frobenius systems} and Corollary \ref{thm: epsilon dependence of g}, but now we allow the case where $A$ and $\tilde{A}$ are merely isomorphic via a Frobenius equivalence.

\begin{lemma}
\label{thm: g under frobenius equivalence}
Assume that $(R,A,\varepsilon, (x_i,y_i))$ and $(R, \tilde{A}, \tilde{\varepsilon}, (\tilde{x}_i,\tilde{y}_i))$ are Frobenius equivalent via $\phi: A \rightarrow \tilde{A}$.  Then $\phi(g) = \tilde{g}$.
\end{lemma}
\begin{proof}
Let $\phi: A \rightarrow \tilde{A}$ be the $R$-linear ring isomorphism such that $\varepsilon = \tilde{\varepsilon} \phi$.  We may assume WLOG that $\phi(x_i)=\tilde{x}_i$, as we have already demonstrated that the genus-reduction term is invariant under change of dual basis.  Recall that we use the matrix $\lambda = [[\varepsilon(x_i x_j)]]$ to determine the $(y_i)$ half of the dual basis.  For our second system above we have $\tilde{\lambda}=[[\tilde{\varepsilon}(\tilde{x}_i \tilde{x}_j)]]=[[\tilde{\varepsilon}(\phi(x_i)\phi(x_j))]]=[[\tilde{\varepsilon}(\phi(x_i x_j))]]=[[\varepsilon(x_i x_j)]]=\lambda$.  Thus if $y_i = \sum_m c_{im} x_m$ in our first system ($c_{im} \in R$), we have the same scalars for $\tilde{y}_i = \sum_m c_{im} \tilde{x}_m$.  Our genus-reduction terms are then $g= \sum_i x_i y_i = \sum_{i,m} c_{im} x_i x_m$ and $\tilde{g} = \sum_i \tilde{x}_i \tilde{y}_i = \sum_{i,m} c_{im} \tilde{x}_i \tilde{x}_m$, from which it follows that $\phi(g)=\tilde{g}$.
\end{proof}

\subsection{Neck-cutting in $sl(n)$ Frobenius Extensions}
\label{subsec: neck-cutting in universal sln}

Let us now direct our attention towards the specific class of universal $sl(n)$ Frobenius extensions that we introduced at the beginning of Section \ref{sec: skein modules}.  With the dual basis that we found in Lemma \ref{thm: dual basis}, we have the genus-reduction term:

\begin{equation}
\label{eq: g}
g = nx^{n-1} - (n-1)a_1x^{n-2} - (n-2)a_2x^{n-3} - ... - a_{n-1}
\end{equation}

Note that this $g$ is merely the derivative $p'(x)$ of the degree-n polynomial from the definition our ring $A=\Z[a_1,...,a_n][x]/(p(x))$.  This will greatly simplify some upcoming calculations. 

Before continuing on, it will prove useful to fully factor $p(x)$ over $\C$ as $p(x)=x^n-a_1x^{n-1}-...-a_{n-1}x-a_n = \prod_{i=1}^n (x + \alpha_i)$.  We may then relate the $a_i$ to the $\alpha_i$ by $a_k = -e_k$, where $e_k$ denotes the k\textsuperscript{th} elementary symmetric polynomial in the $n$ variables $\lbrace \alpha_1, ... \alpha_n \rbrace$.  Note that these succinct equations help to motivate our unconventional decision to write the $\alpha_i$ as the negatives of the roots of $p(x)$ as opposed to the roots themselves, as the later choice would have required the introduction of alternating $(-1)^j$ terms in many of our upcoming results.  Also note that the roots $\alpha_i$ may not all lie in the ring $R=\Z[a_1,...,a_n]=\Z[e_1,...,e_n]$ (although our original coefficients $a_i$ will always lie in the larger ring $\tilde{R}=\Z[\alpha_1,...,\alpha_n]$).  Luckily, despite the fact that a number of our results will depend upon this factorization, and that we will oftentimes need to temporarily pass to the ``more general" Frobenius extension $\tilde{R} \hookrightarrow \tilde{R}[x]/(p(x))$, all of our conclusions will descend back down to our original Frobenius system.

To offer a bit of insight into the general situation, pause to consider the specific case of $n=2$.  Here we have $p(x)=x^2-a_1 x-a_2 = (x+\alpha_1)(x+\alpha_2)$ and hence $g = 2x-a_1 = 2x+(\alpha_1 + \alpha_2)$ by our earlier observation.  It follows that $g^2 = 4 a_2 +a_1^2 = (\alpha_1 - \alpha_2)^2$, so we have $g^2=0$ (and hence $g^i$ for all $ \geq 2$) in $A$ iff the two roots of $p(x)$ coincide.  This result will have a natural extension to higher $n$ that we will address in Theorem \ref{thm: repeated roots}.

At least in the $n=2$ case, the fact that $g^2 \in R$ is a constant also allows us to easily characterize all powers of $g$.  In particular, $g^{2i}= (4 a_2 + a_1^2)^{2i}$ and $g^{2i+1}=(4 a_2 + a_1^2)^{2i} (2x-a_1)$ for all $i \geq 0$.  We also have have $x*g^{2i}= (4 a_2 + a_1^2)^{2i} x$ and $x* g^{2i+1} = (4 a_2 + a_1^2)^{2i} (2x^2-a_1x)=(4 a_2 + a_1^2)^{2i} (a_1 x + 2 a_2)$.  Given our standard Frobenius form $\varepsilon(1)=0$ and $\varepsilon(x)=1$, these results allow us to determine how all (marked) closed genus-$i$ surfaces ($i \geq 1$) evaluate in the Frobenius system.  For $\Sigma_k$ a genus-$k$ surface marked with $1$ and $\dot{\Sigma}_k$ a genus-$k$ surface marked with $x$ we have:
\begin{center}
$\Sigma_{2i}=\varepsilon(g^{2i})=0$ \hspace{1.1in} $\Sigma_{2i+1}=\varepsilon(g^{2i+1})= 2(4 a_2 +a_1^2)^{2i}$\\
$\dot{\Sigma}_{2i}=\varepsilon(x*g^{2i})= (4 a_2 + a_1^2)^{2i}$ \hspace{.38in} $\dot{\Sigma}_{2i+1}=\varepsilon(x*g^{2i+1})=(4 a_2 + a_1^2)^{2i} a_1$
\end{center}

Unfortunately, this extremely elegant result does not fully extend to higher $n$, as we don't typically have $g^i \in R$ for any $i \geq 1$.  For an attempt at tackling this general problem using linear algebra, see Appendix 1.  In order to ensure a relatively simple characterization of the $g^i$ for all $i$, we actually need to impose a condition on our polynomial $p(x)$ akin to what was suggested in the $n=2$ case with $\alpha_1=\alpha_2$.  This brings us to the primary theorem of this section, whose converse we briefly delay:

\begin{theorem}
\label{thm: repeated roots}
Consider the universal $sl(n)$ Frobenius extension $R=\Z[a_1,...,a_n] \hookrightarrow A=\Z[a_1,...,a_n][x]/(p(x))$.  If every root of $p(x)$ is a repeated root, then $g^2 = 0$ in $A$.
\end{theorem}
\begin{proof}
Let $p(x)=\prod_{i=1}^n(x+\alpha_i)$, and assume that each of the $\alpha_i$ is a repeated root.  We temporarily pass to $\tilde{A}=\Z[\alpha_1,...,\alpha_n][x]/(p(x))$ to ensure that $\alpha_i \in \tilde{A}$ for all $i$, and first show that $g^2=0$ in $\tilde{A}$.  Recalling that $g(x)=p'(x)$, by the ordinary product rule for derivatives we have $g=\sum_{i=1}^n \frac{p(x)}{(x+\alpha_i)}$ and $g^2= \sum_{i,j=1}^n \frac{p(x)^2}{(x+\alpha_i)(x+\alpha_j)}=p(x) \sum_{i,j=1}^n \frac{p(x)}{(x+\alpha_i)(x+\alpha_j)}$.  If every root $\alpha_i$ is repeated, every term in $\sum_{i,j=1}^n \frac{p(x)}{(x+\alpha_i)(x+\alpha_j)}$ can be rewritten with denominator $1$ and we see that $p(x)$ divides $g^2$ in $\Z[\alpha_1,...\alpha_n]$.  Thus $g^2 = 0$ in $\tilde{A}$.\\
To prove the stronger statement that $g^2 = 0$ in $A$, we introduce some new notation.  Let $e_k^{\alpha_i \alpha_j}$ denote the k\textsuperscript{th} elementary symmetric polynomial in the $n-2$ roots of $p(x)$ that aren't $\alpha_i$ or $\alpha_j$.  Expanding the degree $n-2$ polynomial $f(x)=\sum_{i,j=1}^n \frac{p(x)}{(x+\alpha_i)(x+\alpha_j)}$ from above, the coefficient of $x^q$ takes the form $c_q = \sum_{i,j=1}^n e_{n-2-q}^{\alpha_i \alpha_j}$.  Each of these $c_q$ is a clearly symmetric polynomial in all of the $\alpha_i$.  Thus by the Fundamental Theorem of Symmetric Functions we know that each of the $c_q$ can be generated by the elementary symmetric polynomials in all of the $\alpha_i$, implying that $f(x)$ is actually in $\Z[e_1,...,e_n]=\Z[a_1,...,a_n]$ and that $g^2 = 0$ in $A$.\\
\end{proof}

Apart from allowing us to quickly evaluate all closed compact surfaces of genus $i \geq 2$, this theorem also implies that any surface with a component admitting multiple (non-separating) compressions in the given type of Frobenius system must evaluate to zero.  This insight will great aid us in Section \ref{sec: embedded skein modules}, when we attempt to give a presentation of skein modules that are embedded within an arbitrary 3-manifold.

We close this section with the converse of Theorem \ref{thm: repeated roots} and a couple of quick corollaries.  This direction of the theorem actually requires a slightly more involved approach, and quite honestly was one that we also could have used above (with a few additional lemmas).  The necessity of the distinct approach is due to the fact that the summation from Theorem \ref{thm: repeated roots} can only be easily reduced to $\sum \frac{p(x)^2}{(x+\alpha_i)^2}$, where the sum is over only the non-repeated roots $\alpha_i$, and that it seems rather difficult to demonstrate that this remaining term is necessarily nonzero in $A$.

Whereas we only needed to briefly switch to the larger ring $\tilde{A}=\Z[\alpha_1,...,\alpha_n][x]/p(x)$ in the proof of Theorem \ref{thm: repeated roots}, the proof of its converse requires that we completely pass to the ``more general" Frobenius extension (with equivalent Frobenius form) over $\tilde{R}=\Z[\alpha_1,...,\alpha_n] \hookrightarrow \tilde{A}$.  This runs against the tradition, followed by Khovanov and others, of adjoining ``just enough" to $\Z$ when defining Frobenius extensions in their development of associated link homologies.  Our departure from convention is justified by the fact that our Frobenius system over $\tilde{R} \hookrightarrow \tilde{A}$ obviously has the same dual basis as the system over $R \hookrightarrow A$, and thus has an identical genus-reduction term $\tilde{g}=g$ that is being reduced mudolo the exact same polynomial $p(x)$.

The following lemma is a straightforward application of the Chinese Remainder Theorem, and explains our reliance upon the ``more general" Frobenius system:

\begin{lemma}
\label{thm: Chinese remainder theorem}
Define $\tilde{R}=\Z[\alpha_1,...\alpha_n]$ as above.  Let $p(x)=\prod_{i=1}^n (x+\alpha_i) = \prod_{i=1}^{m} (x+\alpha_i)^{k_i}$, where $n \geq 2$ and in the second product we have fully grouped like roots.  The Frobenius system over $\tilde{R} \hookrightarrow \tilde{A}=\tilde{R}[x]/(p(x))$, with Frobenius form $\tilde{\varepsilon}(x^{n-1})=1$, $\tilde{\varepsilon}(x^i)=0$ (for $0 \leq i \leq n-2$), is Frobenius equivalent to the Frobenius system over $\tilde{R} \hookrightarrow \hat{A}=\tilde{R}[x]/(x+\alpha_1)^{k_1} \times ... \times \tilde{R}[x]/(x+\alpha_m)^{k_m}$ if we define a Frobenius form on the direct product by $\hat{\varepsilon}(x^{n-1},...,x^{n-1})=1$, $\hat{\varepsilon}(x^i,...,x^i)=0$ (for $0 \leq i \leq n-2$).
\end{lemma}
\begin{proof}
The underlying ring and $R$-linear isomorphism $\phi: \tilde{A} \rightarrow \hat{A}$ follows from the Chinese Remainder Theorem and is given by $\phi(a)=(a,...,a)$.  It is immediate that $\tilde{\varepsilon}=\hat{\varepsilon} \phi$, with the fact that $\tilde{\varepsilon}$ contains no ideals in its nullspace then ensuring the same about $\hat{\varepsilon}$.
\end{proof}

The Frobenius structure that we emplaced on $\tilde{R} \hookrightarrow \hat{A}$ isn't the ``natural" one that brings together the $sl(n)$ systems on each of the coordinates of $\hat{A}$.  In particular, we have done none of the prerequisite work towards determing the dual basis (and hence the genus reduction term) of that system.  The ``natural" Frobenius structure that we want over $\tilde{R} \hookrightarrow \hat{A}$ is the following:
\begin{itemize}
\item Basis $\lbrace (1,0,...,0),...,(x^{k_1-1},0,...,0),(0,1,0,...,0),... \ , \ ...,(0,...,0,x^{k_m-1}) \rbrace$.
\item Frobenius form on that basis given by $\varepsilon'(x^{k_1-1},0,...,0)=\varepsilon'(0,...,0,x^{k_m-1})=1$ and $\varepsilon'(u)=0$ for every other basis element $u$.
\end{itemize}

The Frobenius matrix $\lambda'$ for this system is then block diagonal, with one block for each distinct root $\alpha_i$.  If $\alpha_i$ is a multiplicity one root, its corresponding block is $1 \times 1$ and is the constant matrix $[[1]]$.  If $\alpha_i$ is of multiplicity $n \geq 2$, its block is $n \times n$ with entries identical to the Frobenius matrix $\lambda_i$ for the $sl(n)$ Frobenius system over $R_i=\Z[\alpha_i] \hookrightarrow A_i=\Z[\alpha_i][x]/((x+\alpha_i)^{k_i})$.  As none of these blocks are zero, $\lambda'$ is invertible and $\varepsilon'$ is in fact a nondegenerate Frobenius form.

$(\lambda')^{-1}$ is also block diagonal, with blocks either $[[1]]$ or $(\lambda_i)^{-1}$.  It follows that the dual basis for the above system is $\lbrace ((1,0,...,0),(y_{1,1},0,...,0))$, $...$, $((x^{k_1-1},0,...,0),(y_{1,k_1},0,...,0))$, $((0,1,0,...,0),(0,y_{2,1,0,...,0}))$, $... \ , \ ...$, $((0,...0,x^{k_m-1}),(0,...,0,y_{m,k_m})) \rbrace$, where $y_{i,j}$ is equal to the companion of $x^{j-1}$ in the $sl(n)$ dual basis over $R_i \hookrightarrow A_i$.  The associated genus reduction term is then $g'=(g_1,g_2,...,g_m)$, where $g_i$ is the genus reduction term over $R_i \hookrightarrow A_i$ (with $g_i=1*1=1$ in the coordinates corresponding to multiplicity one roots).  By Theorem \ref{thm: repeated roots}, $(g_i)^2=0$ whenever $\alpha_i$ is a repeated root, while clearly $(g_i)^2=1$ if $\alpha_i$ is multiplicity one.  For this system, it is then obvious that $(g')^2=(g_1^2,g_2^2,...,g_m^2)=0$ iff every root $\alpha_i$ of $p(x)$ is repeated.

We are now ready for the converse of Theorem \ref{thm: repeated roots}:

\begin{theorem}
\label{thm: repeated roots converse}
Let $R= \Z[a_1,...,a_n] \hookrightarrow A= \Z[a_1,...,a_n][x]/(p(x))$ be a $sl(n)$ Frobenius extension ($n \geq 2$), and assume that $p(x)$ has at least one root of multiplicity precisely $1$.  Then $g^2 \neq 0$ in $A$
\end{theorem}
\begin{proof}
Let $p(x)= \prod_{i=1}^m (x+\alpha_i)^{k_i}$, where we have completely grouped like roots, and assume WLOG that $k_1=1$.  We pass to the ``larger" ring extension $\tilde{R} \hookrightarrow \tilde{A}$ described previously, and show that $\tilde{g}=g \neq 0$ in $\tilde{A}$.\\
By Lemma \ref{thm: Chinese remainder theorem}, $\lbrace \tilde{R}, \tilde{A}, \tilde{\varepsilon}, (\tilde{x}_i,\tilde{y}_i) \rbrace$ is Frobenius equivalent to the ``product" Frobenius system $\lbrace \tilde{R}, \hat{A}, \hat{\varepsilon}, (\hat{x}_i,\hat{y}_i) \rbrace$.  By preceding discussion, there exists a Frobenius structure $\lbrace \tilde{R}, \hat{A}, \varepsilon', ({x'}_i, {y'}_i) \rbrace$ over $\tilde{R} \hookrightarrow \hat{A}$ with genus reduction term nonzero.  Corollary \ref{thm: epsilon dependence of g} then ensures that $\hat{g} \neq 0$.  The aforementioned Frobenius equivalence, combined with Lemma \ref{thm: g under frobenius equivalence}, then gives $\tilde{g} \neq 0$.
With $\tilde{g} \neq 0$ in $\tilde{A}$, there cannot exist $\tilde{f}(x) \in \tilde{R}[x]$ such that $p(x)\tilde{f}(x)= \tilde{g}$.  Hence there cannot exist $f(x) \in R[x] \subseteq \tilde{R}[x]$ such that $p(x)f(x) = \tilde{g}=g$, giving $g \neq 0$ in our Frobenius extension over $R \hookrightarrow A$.
\end{proof}

An equivalent argument to Theorem \ref{thm: repeated roots converse} shows that, if $p(x)$ has at least one root of multiplicity $1$, then $g^i \neq 0$ in $A$ for all $i \geq 2$.  Combining results then gives the relatively succinct corollary.

\begin{corollary}
\label{thm: behavior of g^i}
Let $R= \Z[a_1,...,a_n] \hookrightarrow A= \Z[a_1,...,a_n][x]/(p(x))$ be a $sl(n)$ Frobenius extension of rank $n \geq 2$.  If every root of $p(x)$ if repeated, then $g^i=0$ in $A$ for all $i \geq 2$.  Otherwise, $g^i \neq 0$ in $A$ for all $i \geq 2$.
\end{corollary}

This corollary implies that the $sl(n)$ Frobenius extensions associated to $p(x)$ with non-repeated roots have the potential to be extremely complicated, in the sense that they may have closed compact 2-manifolds of arbitrarily high genus that evaluate to nonzero elements of $R$ via $\varepsilon$.

\section{Skein Modules}
\label{sec: embedded skein modules}


In \cite{asaeda}, Asaeda and Frohman explored the free module of isotopy classes of surfaces in a 3-manifold, subject to relations coming from a TQFT over a ring $R$. As in a TQFT, disjoint union behaves like tensor product over $R$. Therefore a surface is viewed as a tensor product of its connected components. The surfaces form a module and this module is an invariant of the 3-manifold the surfaces are embedded in.

The embedded surfaces must be treated slightly differently than the abstract surfaces associated to the TQFT. For instance, it is often the case that the neck-cutting relation cannot be applied as there is no compressing disk present in the 3-manifold. In addition, it is stipulated that the sphere relations only apply to spheres that bound balls. Other than those two considerations, the surfaces are treated as they would be if they are coming from the TQFT. Uwe Kaiser gives a  thorough treatment of obtaining skein modules of 3-manifolds from Frobenius extensions in \cite{kaiser}

In the previous sections we have been dealing with the universal $sl(n)$ Frobenius extensions $R \hookrightarrow A \text{ with } A = R[a_1,\dots,a_n]/(p(x))$, where $p(x)=x^n-a_1x^{n-1} - \dots - a_n$. We now define $K_n(M)$ to be the skein module of $M$ where the surfaces are subject to the relations coming from the general $sl(n)$ Frobenius extension.

Often the $a_i$ are simply indeterminates, but sometimes it is interesting or helpful to examine the skein module where the $a_i$ are subject to certain conditions. This will be indicated by $K_n(M)[\{f_j(a_1,\dots a_n)=0\}_j]$, where the $a_i$ satisfy $f_j(a_1,\dots a_n)=0$, for all $j$.

\subsection{3-manifold Preliminaries}

In order to develop and explore the skein modules, we recall some definitions concerning the study of 3-manifolds.

\begin{definition}

A three-manifold is {\textbf{irreducible}} it every two-sphere bounds a three-ball.

\end{definition}

\begin{definition}

A curve on a surface is {\textbf{inessential}} if it bounds a disk on the surface.  Otherwise the curve is \textbf{essential}.

\end{definition}

\begin{definition}

Let $S$ be a surface embedded in three-manifold $M$.  $S$ is {\textbf{compressible} if $S$ contains an essential curve that bounds a disk, $D$, in $M$ such that $S \cap D = \partial D$.  If no such curves exist and $S$ is not a two-sphere that bounds a ball, then $S$ is \textbf{incompressible}}.

\end{definition}

The compressability of a surface is extremely important when dealing with skein modules. For instance, when a surface is compressible the neck-cutting relation can be applied to yield an equivalent surface in the skein module.

\subsection{Linear Independence of Unmarked Surfaces}

In \cite{asaeda}, Asaeda and Frohman showed that under certain conditions in the $n=2$ case the unmarked surfaces are linearly independent. We extend their result to any $n$ below.

\begin{theorem}
\label{linindtheorem}
Let $M$ be an irreducible three-manifold. If every root of p(x) is repeated, then the unmarked incompressible surfaces in $K_n(M)$ are linearly independent over $\mathbb{Z}[a_1,\dots,a_n]$
\end{theorem}

\begin{proof}
Let $p(x) = \prod_{i=1}^r (x+\alpha_i)^{k_i}$, where $k_i > 1$ for all $i$. Note that $-\alpha_i$ is a root of $g_n= p'(x)$ for all $i$. We show that the unmarked incompressible surfaces of $K_n(M)$ are linearly independent over $\mathbb{Z}[\alpha_1,\dots, \alpha_r]$, implying that they are linearly independent over $\mathbb{Z}[a_1,\dots,a_n] \subset \mathbb{Z}[\alpha_1,\dots, \alpha_r]$

Let $F$ be an unmarked incompressible surface in $M$. For each $F$ we define a $\mathbb{Z}[\alpha_1,\dots, \alpha_r]$-linear functional, $\lambda_F$, such that  $\lambda_F(F) = 1$ and  $\lambda_F(F')=0$ if $F'$ is any other unmarked incompressible surface in $M$. Fix a root $\alpha$ of $p(x)$ (any root will give a suitable family of functionals), and define $\lambda_F$ as follows:

\begin{itemize}

\item $\lambda_F(S) = (-\alpha)^k \prod_\sigma \epsilon(S_\sigma) \prod_\tau \epsilon(T_\tau)$ if $S$ is a disjoint union of $F$, marked with $x^k$, with spheres $S_\sigma$ and compressible tori $T_\tau$.

\item $\lambda_F(S) = \prod_\sigma \epsilon(S_\sigma) \prod_\tau \epsilon(T_\tau)$ if  $S$ is a disjoint union of spheres $S_\sigma$ and compressible tori $T_\tau$.

\item $\lambda_F(S) = 0$ otherwise.

\end{itemize}

Note that since $M$ is irreducible all compressible tori compress down to spheres that bound balls, no matter which compressing disk is chosen.

We must show that the functionals respect the relations of the skein module.  Thus we must address the neck-cutting relation, the sphere relations and the dot reduction relation.

First we show the functionals respect the neck-cutting relation. The functionals are defined so that all surfaces that are compressible (excluding tori) are sent to zero. Therefore we must show that the result of compressing a surface is also sent to zero by the functionals.

By earlier work we have that 

\begin{equation*}
p'(x) =  nx^{n-1} - \sum_{j=0}^{n-2} a_{n-1-j} (j+1) x^{j} 
\end{equation*}

Consider

\begin{align*}
\lambda_F \left (\sum_{i=0}^{n-1}  \raisebox{-.3 in}{\includegraphics[height = .6 in]{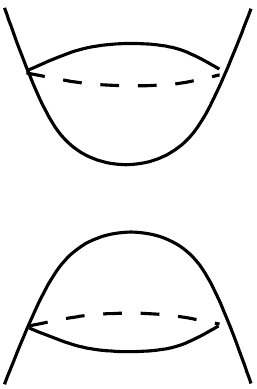}}\put(-18,20){$x^i$}\put(-24,-28){$x^{j-i}$} - \sum_{j=0}^{n-2} a_{n-1-j} \sum_{i=0}^{j}  \raisebox{-.3 in}{\includegraphics[height = .6 in]{cutneck2.pdf}}\put(-18,20){$x^i$}\put(-24,-28){$x^{j-i}$} \right ) &=n(-\alpha)^{n-1} - \sum_{j=0}^{n-2} a_{n-1-j} (j+1) (-\alpha)^{j} \\
&= p'(-\alpha) = 0
\end{align*}

It is also necessary to show that if $x^n$ is replaced by $a_1 x^{n-1} + \dots + a_n$ that the functionals respect this.

First note that

\begin{equation*}
0 = p'(-\alpha) = (-\alpha)^n - a_1 (-\alpha)^{n-1} - \dots - a_n,
\end{equation*}

which implies

\begin{equation*}
(-\alpha)^n =  a_1(- \alpha)^{n-1} + \dots + a_n.
\end{equation*}

Now we have

\begin{align*}
\lambda( a_1(x^{n-1}) + a_2 (x ^{n-2}) + \dots + a_n(1)) &= a_1 (-\alpha)^{n-1} + a_2 (-\alpha)^{n-2} + \dots + a_n  =  (-\alpha)^n \\
\\
&= \lambda(x^n).
\end{align*}

Therefore the functionals respect the ring. Now note the functionals respect the sphere relations by how they are defined.  By the work above they respect the neck-cutting relation. Since we have defined the appropriate functionals it is apparent that the unmarked incompressible surfaces are linearly independent.
\end{proof}

It is important to observe that while we only utilized one repeated root to define our functionals, we actually needed the fact that all roots were repeated. Otherwise we would have needed to address null-homologous  surfaces of genus greater than one in the definition of our functionals. Since all roots are repeated, all higher genus surfaces are equivalent to zero in the skein module, by Corollary \ref{thm: behavior of g^i}.

\subsection{Exploring $K_2(M)$}

In \cite{asaeda}, Asaeda and Frohman  required 2 to be invertible in the ring. Therefore, in order to build on their results, we will work over $\mathbb{Q}[a_1,a_2]$, rather than $\mathbb{Z}[a_1,a_2]$, for the rest of this section.

Recall that in the skein modules the surfaces are marked with elements of the ring. In $K_2(M)$ all surfaces can be written in terms of surfaces that are marked with $x$ to the first power, at most. Following the convention set forth by Bar-Natan, among others, we let a dotted surface represent a surface marked with an $x$.

\subsubsection{An Example}
\label{examplesection}


We will determine the skein module $K_2(S^2 \times S^1)[4a_2 + a_1^2 = 0]$.  Note that $S^2 \times S^1$ is not irreducible so we cannot apply Theorem \ref{linindtheorem}. In order to concisely do calculations in the skein module we introduce some new notation:

\begin{enumerate}

\item \nd will denote an unmarked sphere in $S^2 \times S^1$ that doesn't bound a ball,

\medskip

\item \od will denote a sphere marked with an $x$ in $S^2 \times S^1$ that also doesn't bound a ball,  

\medskip

\item \nd \nd is two parallel unmarked spheres,

\medskip

\item \od \nd is two parallel sphere where one is marked with and $x$ and one is not, etc.

\end{enumerate}

Unless noted otherwise, the spheres are always being viewed locally. That is, there may or may not be other sphere components of the surface in addition to the ones being viewed.

\begin{remark1}

We assume all surfaces are as simple as possible in terms of $x$. If there is an $x^2$ on a component simply replace all of them with $a_1x+a_2$ as follows:

\begin{center}
$\td = a_1 \od + a_2\nd$.
\end{center}

\end{remark1}

\begin{lemma}\label{relationsLemma}

We have the following relations on un-bounding spheres in $S^2 \times S^1$:

\begin{enumerate}

\item \begin{equation*}
\label{relation1}
1 = \od \od + a_2 \nd \nd
\end{equation*}

\item \begin{equation*}
\label{relation2}
\od \nd  = - \nd \od + a_1 \nd \nd
\end{equation*}

\item \begin{equation*}
  \nd\put(-12,-15){1}  \raisebox{.25 em}{\dots} \nd\put(-11,-15){n}  =  \nd\put(-11,-15){n} \nd\put(-11,-15){1}  \raisebox{.25 em}{\dots} \nd\put(-14,-15){n-1}
\end{equation*}

\item \begin{equation*}
\od \od \nd = \nd \od \od 
\end{equation*}

\item \begin{equation*}
\od \nd \nd =  \nd \nd \od
\end{equation*}

\item \begin{equation*}
\od \nd \dots \nd =  \frac{a_1}{2} \nd \nd \dots \nd \text{(exactly one sphere has a dot on left side)}
\end{equation*}

\end{enumerate}

\end{lemma}

\begin{proof}

\begin{enumerate}

\item Consider two parallel spheres that do not bound a ball. If they are tubed together the new sphere bounds a ball, but now compressing the tube we have just placed yields the two parallel spheres. If the sphere that bounds a ball is marked with a dot we get the desired relation.

\item As in 1, if the new sphere is unmarked we get this relation.

\item This relation comes from the fact that we can cyclically permute the spheres in $S^2 \times S^1$.

\item By repeated applications of relation 2 and relation 5, we have:

\begin{equation*}
\od \od \nd = - \od \nd \od + a_1 \od \nd \nd = \nd \od \od - a_1 \nd \nd \od + a_1 \od \nd \nd = \nd \od \od.
\end{equation*}

\item Again, by repeated applications of relation two we have:

\begin{align*}
  \od \nd \nd &= -    \nd \od \nd  + a   \nd \nd \nd =    \nd \nd \od - a  \nd \nd \nd  + a \nd \nd \nd  \\
\\
&=   \nd \nd \od .
\end{align*}

\item

\begin{equation*}
\od \nd \dots \nd = - \nd \od \dots \nd + a_1 \nd \nd \dots \nd 
\end{equation*}

implies that 

\begin{equation*}
2 \od \nd \dots \nd =  a_1 \nd \nd \dots \nd 
\end{equation*}

and since 2 is invertible, we arrive at the desired relation.

\end{enumerate}

\end{proof}

\begin{definition}
An \textbf{odd (even) configuration} of spheres is a surface that consists entirely of an odd (even) number of parallel un-bounding spheres each marked with at most one dot and nothing else.
\end{definition}


\begin{definition}
An even configuration in \textbf{standard position} is in the following form:

\bigskip
\begin{center}
\od \od \dots \od \nd \od \dots \nd \od \nd \nd \dots \nd
\end{center}

An odd configuration in \textbf{standard position} is in the following form:

\bigskip
\begin{center}
\od \od \dots \od \nd \dots \nd
\end{center}

In essence, a configuration in standard position is one where the marked spheres are as close together as possible.

\end{definition}

\begin{lemma}

All configurations can be placed in a unique standard position using relations \ref{relationsLemma}.4 and \ref{relationsLemma}.5.

\end{lemma}
\begin{proof}
First we will consider even configurations. For an even configuration to be in standard position it is necessary to have only one gap between marked spheres of more than one unmarked sphere. If the configuration has two such gaps we can eliminate one by repeated applications of \ref{relationsLemma}.5.  If, after eliminating all such gaps, the configuration is still not in standard position it is because there are adjacent dotted spheres surrounded by spheres where every other sphere is dotted. By applying relation \ref{relationsLemma}.4 we can move all the spheres where each sphere is dotted next to each other and now the configuration is in standard position.

Now we must address uniqueness. Since we are dealing with even configuration we can divide the spheres into two sets, where two spheres are in the same set if they are separated by an odd number of spheres. By the fact we only used relations \ref{relationsLemma}.4 and \ref{relationsLemma}.5 the standard position is completely determined by the number of marked spheres in each set.

For an odd configuration to be in standard position it is necessary to have only one gap between marked spheres. Consider if there are two such gaps. If one of the gaps consists of an even number of unmarked spheres then we can eliminate it by \ref{relationsLemma}.5. Otherwise we can move the dot in the opposite direction by using a combination of \ref{relationsLemma}.4 and \ref{relationsLemma}.5. It is then possible to place it next to a dotted sphere since we are in an odd configuration. Uniqueness follows by the fact the standard position is completely determined by the number of marked spheres.

\end{proof}

\begin{lemma}\label{oddzero}
All odd configurations with no marked spheres are equal to 0.
\end{lemma}

\begin{proof}

\begin{align*}
\nd &= \nd \od \od + a_2 \nd \nd \nd                                           & \text{by relation } \ref{relationsLemma}.1 \\ 
\\
      &= -\od \nd \od + a_1 \nd \nd \od + a_2 \nd \nd \nd & \text{by relation } \ref{relationsLemma}.2 \\
      \\
      &=-(\nd - a_2\nd \nd \nd) + \frac{a_1^2}{2} \nd \nd \nd + a_2 \nd \nd \nd & \text{by relations } \ref{relationsLemma}.2\text{ and }\ref{relationsLemma}.6\\
      \\
      &= - \nd + \frac{a_1^2}{2} \nd \nd \nd+  2a_2 \nd \nd \nd\\
\end{align*}

So, $2 \nd = 2a_2 \nd \nd \nd + \frac{a_1^2}{2} \nd \nd \nd$ and thus

\bigskip

$4 \nd = (4a_2+a_1^2) \nd \nd \nd =0$, since we assumed $4a_2+a_1^2=0$. Thus we have that $\nd = 0$.

\end{proof}

\bigskip

We now define an algorithm for reducing the configurations:

\begin{enumerate}

\item Evaluate all trivial tori and spheres.

\item Using the neck-cutting relation remove all handles from non-bounding spheres.

\item Using relations \ref{relationsLemma}.4 and \ref{relationsLemma}.5 put the configuration into standard position.

\item Using relation \ref{relationsLemma}.2 move all of the dots on parallel spheres as close together as possible.

\item Adjacent dotted spheres annihilate each other by relation \ref{relationsLemma}.1.

\item Using relation \ref{relationsLemma}.6 we are able to replace configurations with one dotted sphere with ones with no dotted spheres.

\end{enumerate}

By applying the algorithm and Lemma \ref{oddzero} we can see that the surfaces are spanned by the collection of even unmarked spheres, one dotted sphere and the empty surface. We wish to show that if $4a_2 + a_1^2=0$ then this collection is linearly independent.

\bigskip

The first step to showing linear independence is to define linear functionals on the generators:

$$
\lambda_k(S) = 
\left \{
\begin{tabular}{cc}
1 & if $S$ is $2k$ parallel unmarked spheres that don't bound a ball \\
0 & else \\
\end{tabular}
\right .
$$

$$
\lambda_d(S) = 
\left \{
\begin{tabular}{cc}
1 & if $S$ is a dotted sphere that doesn't bound a ball \\
0 & else \\
\end{tabular}
\right .
$$

Using the algorithm each linear functional can be extended to a map on any surface in $S^2 \times S^1$. We must show the functionals together with the algorithm are well-defined on the skein module, that is to say that they respect the relations. 

By how the functionals are defined it is clear that they respect the sphere relations. Now we must show that both sides of the neck-cutting relation are respected by the functionals. By the definition of the algorithm, the functionals behave well with regards to neck-cutting, with the exception of when a trivial sphere becomes two non-bounding spheres. These situations are related to relations \ref{relationsLemma}.1 and \ref{relationsLemma}.2 and we address them below.


\medskip

Relation 1: Consider the case where one side of the neck-cutting relation is a bounding unmarked sphere and the other side is the result of compressing the sphere to yield:

\begin{center}
$\od \nd + \nd \od - a_1 \nd \nd$
\end{center}

If there is an even number of spheres then note that either \od \nd or \nd \od will need to be moved to put the dots as close together as possible. Without loss of generality we have

\begin{center}
$\od \nd + \nd \od - a_1 \nd \nd = -\nd \od + a_1 \nd \nd + \nd \od - a_1 \nd \nd= 0$
\end{center}

If there is an odd number of spheres then by Lemma \ref{oddzero} all terms are zero.

\begin{center}
$\od \nd + \nd \od - a_1 \nd \nd = 2\nd \od + a_1 \nd \nd=0$
\end{center}


\medskip

Relation 2 : Consider the case where one side of the neck-cutting relation is a bounding marked sphere and the other is the result of compressing the sphere:

\begin{center}
$\od \od + a_2 \nd \nd$
\end{center}

The first summand has one more pair, so at some point the pair is replaced by

\begin{center}
$(1 - a_2\nd \nd) + a_2 \nd \nd = 1= \text{evaluation of marked sphere} $
\end{center}

Thus we have that 

\begin{equation*}
K_2(S^2 \times S^1)[4a_2 + a_1^2 = 0] \cong R[x] \oplus Re,
\end{equation*}

where $x^k$ represents $2k$ parallel unmarked spheres and $e$ represents a single marked sphere.

\subsubsection{A Partial Converse}

We were able to prove the linear independence of the unmarked incompressible surfaces of $K_n(M)$ when $M$ is irreducible for any $n$, as long as every root of $p(x)$ is repeated. When $n=2$ we are able to prove a partial converse.

In order to prove the converse we use notation similar to that of Section \ref{examplesection}. We will be dealing with an incompressible surface that fibers over a circle. A vertical line will denote one copy of this surface and multiple lines will denote multiple surfaces. If the lines are decorated with a dot than that particular surface is marked with an $x$.

\begin{theorem}\label{irredclass}

Let $M$ be an irreducible 3-manifold such that some incompressible surface in $M$ fibers over a circle.  The unmarked surfaces in $K_2(M)$ are linearly independent if and only if $4a_2+a_1^2 = 0$.

\end{theorem}

\begin{proof}

Right to left is by Theorem \ref{linindtheorem}.  We will show the contrapositive of left to right by showing that if $4a_2+a_1^2 \neq 0$, then the unmarked surfaces are linearly dependent.

Let $i$ be the genus of the incompressible surface that fibers over a circle in $M$. Recall the following results from Section \ref{subsec: neck-cutting in universal sln}:

\begin{itemize}

\item 
${\Sigma}_{2i} = \text{\{Two parallel genus $i$ surfaces tubed together\}} = \od \nd + \nd \od - a_1 \nd \nd$,

\item 
$\dot{\Sigma}_{2i} = \text{\{Two parallel genus $i$ surfaces tubed together, marked with $x$\}} = \od \od + a_2 \nd \nd$.

\end{itemize}

\medskip

Also, note $0 = \Sigma_{2i} = \od \nd + \nd \od - a_1 \nd \nd$, which yields the relation

\begin{equation}\label{convrelation}
a_1 \nd \nd = \od \nd + \nd \od
\end{equation}

By repeated applications of relation \ref{convrelation} above, we have

\begin{equation*}\begin{split}
a_1^2 \nd \nd \nd  &  = a_1 \nd \nd \od + a_1 \nd \od \nd = 2 a_1 \nd \nd \od = 2 \od \nd \od + 2 \nd \od \od \\
\\
&= 4 \nd \od \od = 4 \dot{\Sigma}_{2i} \nd - 4a_2 \nd \nd \nd, 
\end{split}
\end{equation*}

thus, $(4a_2+a_1^2)\nd\nd\nd =4\dot{\Sigma}_{2i} \nd = 4(4 a_2 + a_1^2)^{2i}\nd$, so the unmarked surfaces are linearly dependent.


\end{proof}

Theorem \ref{irredclass} is the only partial converse that we were able to prove for any $n$. Thus, it is an open question as to exactly when the unmarked surfaces are linearly independent for $n > 2$.











\section{Appendix: The Genus Reduction Matrix}
\label{sec: genus reduction matrix}

In Subsection \ref{subsec: neck-cutting in universal sln} we explicitly calculated all powers of our genus reduction term $g$ in the universal $sl(n)$ skein module when $n=2$, and alluded to the fact that this was difficult to do in complete generality for higher $n$.  Here we tackle that problem using linear algebra, interpreting $g$ as a $R$-linear operator from $A$ to $A$.  Choosing the standard ordered basis $\lbrace 1,x,...,x^{n-1} \rbrace$, we may  write $g$ as an $n \times n$ matrix $G_n \in Mat_n(R)$ (where the subscript in $G_n$ corresponds to the rank of the $sl(n)$ extension).

Note that, in terms of our chosen basis, the first column of $G_n$ directly corresponds to our genus-reduction term.  The j\textsuperscript{th} column similarly corresponds to $x^{i-1} g_n$, after reducing modulo $p(x)=x^n - a_1 x^{n-1} - ... - a_{n-1} x - a_n$.  Also note that what our closed surfaces actually evaluate to via our Frobenius form correspond to the final row of the matrix, so that we can immediately determine the evaluation of a torus decorated by $x^k$ as the $(n,k+1)$ entry of $G$.  Our first proposition gives a recursive formula for determining the $(i,j)$ entry of $G_n$, for any $n \geq 2$.

\begin{proposition}
\label{thm: Gn recursive formula}
The entries of $G_n = \left[ \left[ g_{i,j} \right] \right]$, for any $n \geq 2$ are defined recursively as follows:\\
\begin{center}
$g_{i,1} = -i a_{n-i}$ (for $i < n$)\\
$g_{n,1} = n$\\
$g_{i,j} = a_{n-i+1} g_{n,j-1} + g_{i-1,j-1}$ (for $j > 1, i>1$)\\
$g_{i,j} = a_{n-i+1} g_{n,j-1}$ (for $j > 1, i=1$)
\end{center}
\end{proposition}
\begin{proof}
The first two lines follow from the expression for $g$ that we already exhibited at the beginning of Subsection \ref{subsec: neck-cutting in universal sln}.  As for the last two lines, we obtain the j\textsuperscript{th} column of $G_n$ from the $(j-1)^{th}$ column via multiplication by $x$.  Working modulo $(x^n - a_1 x^{n-1} - ... - a_n)$ we then have:\\
$x*(g_{1,j-1} + g_{2,j-1}x + ... + g_{n-1,j-1}x^{n-2} + g_{n,j-1}x^{n-1}) = g_{1,j-1}x + g_{2,j-1}x^2 + ... + g_{n-1,j-1}x^{n-1} + g_{n,j-1}x^{n} = g_{1,j-1}x + g_{2,j-1}x^2 + ... + g_{n-1,j-1}x^{n-1} + g_{n,j-1}(a_1 x^{n-1} + a_2 x^{n-2} + ... + a_{n-1} x + a_n) = (a_n g_{n,j-1}) + (a_{n-1} g_{n,j-1} + g_{1,j-1}) x + ... + (a_1 g_{2,j-1} + g_{n-1,j-1}) x^{n-1}$
\end{proof}

When our recursive relation it is then easy to produce $G_n$ for small $n$:

\[ G_2 =
\begin{bmatrix}
-a_1 & 2 a_2 \\
2 & a_1
\end{bmatrix}
\]

\[ G_3 =
\begin{bmatrix}
-a_2 & 3 a_3 & a_1 a_3 \\
-2 a_1 & 2 a_2 & a_1 a_2 + 3 a_3 \\
3 & a_1 & a_1^2 + 2 a_2
\end{bmatrix}
\]

\[ G_4 =
\begin{bmatrix}
-a_3 & 4 a_4 & a_1 a_4 & a_1^2 a_4 + 2 a_2 a_4 \\
-2 a_2 & 3 a_3 & a_1 a_3 + 4 a_4 & a_1^2 a_3 + 2 a_2 a_3 + a_1 a_4\\
-3 a_1 & 2 a_2 & a_1 a_2 + 3 a_3 & a_1^2 a_2 + 2 a_2^2 + a_1 a_3 + 4 a_4\\
4 & a_1 & a_1^2 + 2 a_2 & a_1^3 + 3 a_1 a_2 + 3 a_3
\end{bmatrix}
\]
\\
The relatively simple conclusions that we drew about the $n=2$ case in Subsection \ref{subsec: neck-cutting in universal sln} follow directly from the fact that:
\[ (G_2)^2 =
\begin{bmatrix}
a_1^2 + 4 a_2 & 0 \\
0 & a_1^2 + 4 a_2
\end{bmatrix}
=
(a_1^2 + 4 a_2)*E_2 
\]
And hence that:
\[
(G_2)^{2k} =
\begin{bmatrix}
(a_1^2 + 4 a_2)^k & 0 \\
0 & (a_1^2 + 4 a_2)^k
\end{bmatrix}
\]
\[
(G_2)^{2k+1} =
\begin{bmatrix}
-a_1 (a_1^2 + 4 a_2)^k & 2 a_2 (a_1^2 + 4 a_2)^k \\
2 (a_1^2 + 4 a_2)^k & a_1 (a_1^2 + 4 a_2)^k
\end{bmatrix}
\]

Now recall our complete factorization of $p(x)$ over $\C$ as $p(x)=x^n-a_a x^{n-1}-...-a_{n-1}x-a_n = \prod_{i=1}^n(x+\alpha_i)$, which provides for the identification of $a_k$ with the (negative of the) k\textsuperscript{th} elementary syymetric polynomial $e_k$ in the $\alpha_i$.  When hoping to rewrite our matrices $G_n$ is terms of the $\alpha_i$, we require more general symmetric polynomials than the elementary ones.  Hence we introduce the monomial symmetric polynomials, with $m_{(k_1...k_n)}$ standing for the sum of all monomials in the $\alpha_i$ of the form $\alpha_{i_2}^{k_2}...\alpha_{i_n}^{k_n}$.  Note that we have as special subcases the elementary symmetric polynomials $e_k = m_{(1^k 0^{n-k})} = m_{(1^k)}$, where the $1^k$ indicates $k$ consecutive 1's and we traditionally drop any trailing 0's for brevity.  In this notation we also have the ``power" symmetric polynomials $p_k = m_{(k^1)}$.

We may then quickly rewrite the first several $G_n$ from above:

\[ G_2 =
\begin{bmatrix}
m_{(1^1)} & -2 m_{(1^2)} \\
2 & -m_{(1^1)}
\end{bmatrix}
\]

\[ G_3 =
\begin{bmatrix}
m_{(1^2)} & -3 m_{(1^3)} & m_{(2^1 1^2)} \\
2 m_{(1^1)} & -2 m_{(1^2)} & m_{(2^1 1^1)} \\
3 & -m_{(1^1)} & m_{(2^1)}
\end{bmatrix}
\]

\[ G_4 =
\begin{bmatrix}
m_{(1^3)} & -4 m_{(1^4)} & m_{(2^1 1^3)} & -m_{(3^1 1^3)}\\
2 m_{(1^2)} & -3 m_{(1^3)} & m_{(2^1 1^2)} & -m_{(3^1 1^2)}\\
3 m_{(1^1)} & -2 m_{(1^2)} & m_{(2^1 1^1)} & -m_{(3^1 1^1)}\\
4 & -m_{(1^1)} & m_{(2^1)} & -m_{(3^1)}
\end{bmatrix}
\]

There's an obvious pattern for $G_n$ that begins to emerge here, and that pattern becomes especially simple following the first two columns.  To show that this pattern holds for all $n \geq 2$ we require the following basic properties of symmetric polynomials, all of which are directly verifiable:

\begin{lemma}
\label{thm: symmetric polynomial relations}
For $p_a = m_{(a^1)}$ and $e_b = m_{(1^b)}$ in $n$ variables, we have the following relations:
\begin{enumerate}
\item $p_a e_b = m_{((a+1)^1 1^{n-1})}$ (for $b=n$)
\item $p_a e_b = m_{(2^1 1^{b-1})}+(b+1)m_{(1^{b+1})}$ (for $a=1$ and $b<n$)
\item $p_a e_b = m_{((a+1)^1 1^{b-1})}+m_{(a^1 1^b)}$ (for $a >1$ and $b<n$)
\end{enumerate}
\end{lemma}

\begin{proposition}
\label{thm: symmetric polynomial matrix}
For any $n \geq 2$, $G_n$ is of the form:\\
\[ G_n =
\begin{bmatrix}
m_{(1^{n-1})} & -n m_{(1^n)} & m_{(2^1 1^{n-1})} & -m_{(3^1 1^{n-1})} & \ldots & (-1)^{n-1} m_{((n-1)^1 1^{n-1})}\\
2m_{(1^{n-2})} & -(n-1) m_{(1^{n-1})} & m_{(2^1 1^{n-2})} & -m_{(3^1 1^{n-2})} & \ldots & (-1)^{n-1} m_{((n-1)^1 1^{n-2})}\\
\vdots & \vdots & \vdots & \vdots & \vdots & \vdots \\
(n-1)m_{(1^1)} & -2 m_{(1^2)} & m_{(2^1 1^1)} & -m_{(3^1 1^1)} & \ldots & (-1)^{n-1} m_{((n-1)^1 1^1)}\\
n & -m_{(1^1)} & m_{(2^1)} & -m_{(3^1)} & \ldots & (-1)^{n-1} m_{((n-1)^1)}
\end{bmatrix}
\]
\end{proposition}
\begin{proof}
We use the recursive relations proven in Proposition \ref{thm: Gn recursive formula}. As the pattern stabilizes beginning with the third column, we use those relations to directly verify the entries of columns $j=1$ and $j=2$, and then use induction for columns $j \geq 3$.\\
Column $j=1$:\\
$g_{i,1} = -i a_{n-i}=i e_{n-i} = i m_{(1^{n-i})}$ (for $i<n$)\\
$g_{n,1} = n$\\
Column $j=2$:\\
$g_{1,2} = a_n g_{n,1} = n a_n = -n e_n = -n m_{(1^1)}$\\
$g_{i,2} = a_{n-i+1} g_{n,1} + g_{i-1,1} = -n e_{n-i+1} + (i-1) e_{n-i+1} = -(n-i+1)e_{n-i+1}$ (for $i>1$)\\
Column $j=3$ (inductive base step), noting that $g_{n,2} = -p_1$:\\
$g_{1,3} = a_n g_{n,2} = -e_n (-p_1) = m_{(2^1 1^{n-1})}$ by Lemma \ref{thm: symmetric polynomial relations}(1)\\
$g_{i,3} = a_{n-i+1} g_{n,2} + g_{i-1,2} = -e_{n-i+1} (-p_1 ) + -(n-i+2) e_{n-i+2} = m_{(2^1 1^{n-i})} + (n-i+2) m_{(1^{n-i+2})} - (n-i+2) m_{(1^{n-i+2})} = m_{(2^1 1^{n-i})}$ by Lemma \ref{thm: symmetric polynomial relations}(2) (for $i>1$)\\
Inductive step (assume pattern holds for column $k$), noting that $g_{n,k} = (-1)^{k-1} p_{k-1}$:\\
$g_{1,k+1} = a_n g_{n,k} = -e_n (-1)^{k-1} p_{k-1} = (-1)^k m_{(k^1 1^{n-1})}$ by Lemma \ref{thm: symmetric polynomial relations}(1)\\
$g_{i,k+1} = a_{n-i+1} g_{n,k} + g_{i-1,k} = -e_{n-i+1} (-1)^{k-1} p_{k-1} +(-1)^{k-1} m_{((k-1)^1 1^{n-i+1})} = (-1)^k m_{(k^1 1^{n-i})} + (-1)^k m_{((k-1)^1,1^{n-i+1})} + (-1)^{k-1} m_{((k-1)^1 1^{n-i+1})} = (-1)^k m_{(k^1 1^{n-i})}$ by Lemma \ref{thm: symmetric polynomial relations}(3)\\
\end{proof}











\nocite{*}

\bibliographystyle{plain}
\bibliography{GeneralizedSkeinModulesofSurfaces}

\end{document}